\documentclass[12pt]{article}
\usepackage[utf8]{inputenc}
\usepackage{subfig}
\usepackage{amsmath, amsthm, amssymb}
\usepackage{hyperref}
\usepackage[margin=1cm]{caption}
\usepackage{verbatim}
\usepackage{tikz,tkz-graph}
\usepackage[top=1.0in, bottom=1.0in, left=1.0in, right=1.0in]{geometry}

\usepackage{sectsty}
\allsectionsfont{\sffamily}

\theoremstyle{plain}
\newtheorem{thm}{Theorem}
\newtheorem*{mainthm1}{Theorem~1}
\newtheorem*{mainthm2}{Theorem~2}
\newtheorem*{VAL}{Vizing's Adjacency Lemma}
\newtheorem{prop}{Proposition}
\newtheorem{lem}[prop]{Lemma}

\newtheorem{cor}[prop]{Corollary}

\newtheorem{example}{Example}
\newtheorem{clm}{Claim}

\theoremstyle{definition}

\theoremstyle{remark}

\newcommand{\set}[1]{\left\{ #1 \right\}}

\def\ch{{\rm ch}}

\def\hajos{Haj\'{o}s}
\tikzstyle{labeledThree}=[shape = circle, minimum size = 6pt, inner sep = 1.2pt, draw]
\tikzstyle{labeledTwo}=[shape = rectangle, minimum size = 8pt, inner sep = 2.2pt, draw]

\newcommand{\GG}[1]{}

\author{Daniel W. Cranston\thanks{Department of Mathematics and Applied
Mathematics, Viriginia Commonwealth University, Richmond, VA;
\texttt{dcranston@vcu.edu}; 
The first author's research is partially supported by NSA Grant
H98230-15-1-0013.}
\and
Landon Rabern\thanks{LBD Data Solutions, Lancaster, PA;
\texttt{landon.rabern@gmail.com}}
}
\title{Subcubic edge-chromatic critical graphs\\ have many edges}

\begin{document}
\maketitle
\begin{abstract}
We consider graphs $G$ with $\Delta=3$ such that $\chi'(G)=4$ and $\chi'(G-e)=3$ for
every edge $e$, so-called \emph{critical} graphs.  Jakobsen noted that the
Petersen graph with a vertex deleted,
$P^*$, is such a graph and has average degree only $\frac83$.  He showed
that every critical graph has average degree at least $\frac83$, and asked if
$P^*$ is the only graph where equality holds.  A result of Cariolaro and
Cariolaro shows that this is true.  We strengthen this average degree bound further.
Our main result is that if $G$ is a subcubic
critical graph other than $P^*$, then 
$G$ has average degree at least
$\frac{46}{17}\approx2.706$.
This bound is best possible, as shown by the \hajos\ join of two copies of
$P^*$.
\end{abstract}
\section{Introduction}

By a \emph{coloring} of a graph $G$, we mean
a \emph{proper edge-coloring}, which assigns a color to each edge in
$E(G)$ such that edges with a common endpoint receive distinct colors.  
The minimum number of colors needed for a proper edge-coloring is the
\emph{edge-chromatic number of $G$}, denoted $\chi'(G)$.  
Given a (partial) coloring of $G$,
if a color $c$ is used on an edge incident to a vertex $v$, then 
$v$ \emph{sees} $c$; otherwise $v$ \emph{misses} $c$.
The maximum degree of
$G$ is denoted $\Delta(G)$, or simply $\Delta$ when $G$ is clear from context. 
Its \emph{average degree}, $2|E(G)|/|V(G)|$, is denoted $a(G)$.
Note that always $\Delta(G)\le \chi'(G)$.  Vizing famously proved that always
$\chi'(G)\le \Delta(G)+1$.  A graph is
\emph{edge-chromatic critical} (also \emph{$\Delta$-critical}, or simply
\emph{critical}) if $\chi'(G)>\Delta(G)$ but $\chi'(G-e)=\Delta(G)$ for every
edge $e$.  
A vertex of degree $k$ is a \emph{$k$-vertex}.  
If $v$ is a $k$-vertex and $v$ is adjacent to $u$, then $v$ is a
\emph{$k$-neighbor} of $u$.  The \emph{order} of $G$ is $|V(G)|$.

Vizing~\cite{Vizing65,Vizing68} was the first to seek a lower bound on the
number of edges in a
critical graph, in terms of its order.  This problem is now widely
studied, for a large range of maximum degrees $\Delta$.  Woodall gives
a nice history of this work, in the introduction to~\cite{Woodall08}.
In this paper, we study the problem for subcubic graphs, i.e., when $\Delta=3$.

It is easy to check that
2-critical graphs are precisely odd cycles, which are completely understood.  So
the first non-trivial case is 3-critical graphs.
Let $P^*$ denote the Petersen graph with a vertex deleted.
Jakobsen~\cite{Jakobsen73} observed that $P^*$ is 3-critical and has average
degree $\frac83$.
In the same paper he asked if every 3-critical graph $G$ has $a(G)\ge \frac83$. 
A year later~\cite{Jakobsen74}, he answered this question
affirmatively. 
However, in this second paper Jakobsen asked whether 
$P^*$ is the only graph for which the bound holds with equality.
The answer is yes.  This follows easily from a result of Cariolaro and
Cariolaro~\cite{C&C}, as we now show.

\begin{prop}
If $G$ is a 3-critical graph of order $n$, then $|E(G)|\ge \frac43n$ and
equality holds only if $G=P^*$.
\end{prop}
\begin{proof}
Let $G$ be a 3-critical graph of order $n$.  Let $n_k$ denote the number of
$k$-vertices in $G$, for each $k\in\{2,3\}$; let $G^{(3)}$ denote the subgraph
induced by all 3-vertices.  An easy consequence of Vizing's Adjacency Lemma (see
the start of Section~\ref{sec:discharging}) is
that each 2-vertex is adjacent to two 3-vertices and each 3-vertex is adjacent
to at most one 2-vertex.  Thus $n_3\ge 2n_2$ and equality implies
$\Delta(G^{(3)})\le 2$.  Letting $m=|E(G)|$, we get
$$
\frac{2m}{n}=\frac1n(2n_2+3n_3)=2+\frac{n_3}n\ge \frac83.
$$
The last equality is equivalent to $\frac{n_3}n\ge \frac23$.  This is equivalent
to $3n_3\ge 2n=2(n_2+n_3)$, and hence to $n_3\ge 2n_2$.  So if
$\frac{2m}n=\frac83$, then $n_3=2n_2$, and thus $\Delta(G^{(3)})\le 2$.  Now a
result of Cariolaro and Cariolaro~\cite{C&C} (see
also~\cite[Theorem~4.11]{StiebitzSTF12}) implies that $G=P^*$.
\end{proof}

A natural extension of this question is
to find the maximum $\alpha$ such that every 3-critical graph other than
$P^*$ has average degree at least $2+\alpha$.  A complementary question 
is to find the minimum $\beta$ such that
there exists an infinite sequence of 3-critical
graphs with average degree at most $2+\beta$.  The first progress toward
answering this question is due to Fiorini and Wilson~\cite[p. 43]{FioriniWilson},
who constructed an infinite family of 3-critical graphs with average degree
approaching $2+\frac34$ from below.  
Woodall~\cite[p. 815]{Woodall08} gave another family with the same number of
edges and vertices; see Figure~\ref{woodall-examples}.
Before presenting this construction, we need a definition.

Let $G_1$ and $G_2$ be two graphs with $v_1v_2\in E(G_1)$ and $v_3v_4\in E(G_2)$.
A \emph{\hajos\ join} of $G_1$ and $G_2$ is formed from the disjoint union of
$G_1-v_1v_2$ and $G_2-v_3v_4$ by identifying vertices $v_1$ and $v_3$ and
adding the edge $v_2v_4$.

\begin{lem}
If $G_1$ and $G_2$ are $k$-critical graphs, and $G$ is a \hajos\ join of $G_1$ and
$G_2$ that has maximum degree $k$, then $G$ is also $k$-critical.
\label{critical}
\end{lem}

This is an old result of Jakobsen~\cite{Jakobsen73}.  It is a
straightforward exercise, so we omit the details,  which are available in
Fiorini \& Wilson~\cite[p. 82--83]{FioriniW77} and Stiebitz et al.~\cite[p.
94]{StiebitzSTF12}.

\begin{cor}
Let $G_1$ and $G_2$ be subcubic graphs, and let $G_1$ be 3-critical.  If $G$ is
a subcubic graph that is a \hajos\ join of $G_1$ and $G_2$, then $G$ is
3-critical if and only if $G_2$ is 3-critical.
\label{hajos-cor}
\end{cor}
\begin{proof}
The ``if'' direction follows immediately from the previous lemma.

To prove the ``only if'' direction, assume that $G_2$ is not 3-critical. Either
$\chi'(G_2)=4$ or $\chi'(G_2)\le 3$;
first assume that $\chi'(G_2)=4$.  Since $G_2$ is not 3-critical,
there exists $e\in E(G_2)$ such that $\chi'(G_2-e)=4$.
Suppose that $\chi'(G_2-v_3v_4)=4$.  Now $\chi'(G)\ge 4$, since
$G_2-v_3v_4\subseteq G$.  Further, $G$ is not 3-critical, since
$G_2-v_3v_4\subseteq G-e$ for every $e\in E(G_1)-v_1v_2$.  So suppose
there exists a 3-critical subgraph $J\subsetneq G_2$ such that $v_3v_4\in E(J)$.
By Lemma~\ref{critical}, the \hajos\ join of $G_1$ and $J$ is 3-critical; but
this is a proper subgraph of $G$, so $G$ is not 3-critical.

Instead we assume that $\chi'(G_2)\le3$.
Let $C_1$ be a 3-coloring of $G_1-v_1v_2$ (which exists, by criticality); by
symmetry among colors, assume
that $v_2$ misses color $x$.  Note that $v_1$ must see $x$, since
otherwise we get a 3-coloring of $G_1$.  If $d_{G_1-v_1v_2}(v_1)=2$, then assume
also that $v_1$ sees color $y$.
Let $C_2$ be a 3-coloring of $G_2$. By symmetry, assume that $v_3v_4$ uses color
$x$; if $d_{G_2}(v_3)=2$, then assume also by symmetry that $v_3$ sees color
$z$.  To get a 3-coloring of $G$, use $C_1$ on $G-v_1v_2$, use $C_2$ on
$G_2-v_3v_4$, and use color $x$ on $v_2v_4$.
\end{proof}

Now we present a construction of 3-critical graphs with few edges; see
\cite[p.~43]{FioriniWilson} and \cite[p.~815]{Woodall08}.

\begin{example}
Form $J_k$ by starting with $P^*$ and taking the \hajos\ join
with $P^*$ a total of $k$ times (successively), so that each intermediate graph
has $\Delta=3$.  The resulting graph $J_k$ is 3-critical, has $11k+12$ edges and
$8k+9$ vertices.  Thus, $a(J_k)\to \frac{11}4$ from below as $k\to \infty$.
\label{example1}
\end{example}

The vertex and edge counts follow immediately by induction.  That $J_k$ is
3-critical uses induction and also Lemma~\ref{critical}. $\qed$

\begin{figure}
\begin{center}
\begin{tikzpicture}[scale = 6]
\tikzstyle{VertexStyle} = []
\tikzstyle{EdgeStyle} = []
\tikzstyle{labeledStyle}=[shape = circle, minimum size = 6pt, inner sep = 1.2pt, draw]
\tikzstyle{unlabeledStyle}=[shape = circle, minimum size = 6pt, inner sep =
1.2pt, draw, fill]
\Vertex[style = unlabeledStyle, x = 0.45, y = 0.80, L = \tiny {}]{v0}
\Vertex[style = unlabeledStyle, x = 0.35, y = 0.75, L = \tiny {}]{v1}
\Vertex[style = unlabeledStyle, x = 0.30, y = 0.65, L = \tiny {}]{v2}
\Vertex[style = unlabeledStyle, x = 0.35, y = 0.55, L = \tiny {}]{v3}
\Vertex[style = unlabeledStyle, x = 0.45, y = 0.50, L = \tiny {}]{v4}
\Vertex[style = unlabeledStyle, x = 0.55, y = 0.55, L = \tiny {}]{v5}
\Vertex[style = unlabeledStyle, x = 0.60, y = 0.65, L = \tiny {}]{v6}
\Vertex[style = unlabeledStyle, x = 0.55, y = 0.75, L = \tiny {}]{v7}
\Vertex[style = unlabeledStyle, x = 0.45, y = 0.95, L = \tiny {}]{v8}
\Edge[label = \tiny {}, labelstyle={auto=right, fill=none}](v1)(v2)
\Edge[label = \tiny {}, labelstyle={auto=right, fill=none}](v1)(v0)
\Edge[label = \tiny {}, labelstyle={auto=right, fill=none}](v7)(v0)
\Edge[label = \tiny {}, labelstyle={auto=right, fill=none}](v7)(v6)
\Edge[label = \tiny {}, labelstyle={auto=right, fill=none}](v7)(v3)
\Edge[label = \tiny {}, labelstyle={auto=right, fill=none}](v5)(v6)
\Edge[label = \tiny {}, labelstyle={auto=right, fill=none}](v5)(v4)
\Edge[label = \tiny {}, labelstyle={auto=right, fill=none}](v5)(v1)
\Edge[label = \tiny {}, labelstyle={auto=right, fill=none}](v3)(v4)
\Edge[label = \tiny {}, labelstyle={auto=right, fill=none}](v3)(v2)
\path [line width=1pt]
(v2) edge [bend left=35] (v8)
(v8) edge [bend left=35] (v6);
\end{tikzpicture}
~~
\begin{tikzpicture}[scale = 6]
\tikzstyle{VertexStyle} = []
\tikzstyle{EdgeStyle} = []
\tikzstyle{labeledStyle}=[shape = circle, minimum size = 6pt, inner sep = 1.2pt, draw]
\tikzstyle{unlabeledStyle}=[shape = circle, minimum size = 6pt, inner sep = 1.2pt, draw, fill]
\Vertex[style = unlabeledStyle, x = 0.85, y = 0.80, L = \tiny {}]{v0}
\Vertex[style = unlabeledStyle, x = 0.75, y = 0.75, L = \tiny {}]{v1}
\Vertex[style = unlabeledStyle, x = 0.70, y = 0.65, L = \tiny {}]{v2}
\Vertex[style = unlabeledStyle, x = 0.75, y = 0.55, L = \tiny {}]{v3}
\Vertex[style = unlabeledStyle, x = 0.85, y = 0.50, L = \tiny {}]{v4}
\Vertex[style = unlabeledStyle, x = 0.95, y = 0.55, L = \tiny {}]{v5}
\Vertex[style = unlabeledStyle, x = 1.00, y = 0.65, L = \tiny {}]{v6}
\Vertex[style = unlabeledStyle, x = 0.95, y = 0.75, L = \tiny {}]{v7}
\Vertex[style = unlabeledStyle, x = 0.45, y = 0.80, L = \tiny {}]{v8}
\Vertex[style = unlabeledStyle, x = 0.35, y = 0.75, L = \tiny {}]{v9}
\Vertex[style = unlabeledStyle, x = 0.30, y = 0.65, L = \tiny {}]{v10}
\Vertex[style = unlabeledStyle, x = 0.35, y = 0.55, L = \tiny {}]{v11}
\Vertex[style = unlabeledStyle, x = 0.45, y = 0.5, L = \tiny {}]{v12}
\Vertex[style = unlabeledStyle, x = 0.55, y = 0.55, L = \tiny {}]{v13}
\Vertex[style = unlabeledStyle, x = 0.60, y = 0.65, L = \tiny {}]{v14}
\Vertex[style = unlabeledStyle, x = 0.55, y = 0.75, L = \tiny {}]{v15}
\Vertex[style = unlabeledStyle, x = 0.65, y = 0.95, L = \tiny {}]{v16}
\Edge[label = \tiny {}, labelstyle={auto=right, fill=none}](v1)(v0)
\Edge[label = \tiny {}, labelstyle={auto=right, fill=none}](v1)(v2)
\Edge[label = \tiny {}, labelstyle={auto=right, fill=none}](v3)(v2)
\Edge[label = \tiny {}, labelstyle={auto=right, fill=none}](v3)(v4)
\Edge[label = \tiny {}, labelstyle={auto=right, fill=none}](v5)(v1)
\Edge[label = \tiny {}, labelstyle={auto=right, fill=none}](v5)(v4)
\Edge[label = \tiny {}, labelstyle={auto=right, fill=none}](v5)(v6)
\Edge[label = \tiny {}, labelstyle={auto=right, fill=none}](v7)(v0)
\Edge[label = \tiny {}, labelstyle={auto=right, fill=none}](v7)(v3)
\Edge[label = \tiny {}, labelstyle={auto=right, fill=none}](v7)(v6)
\Edge[label = \tiny {}, labelstyle={auto=right, fill=none}](v9)(v8)
\Edge[label = \tiny {}, labelstyle={auto=right, fill=none}](v9)(v10)
\Edge[label = \tiny {}, labelstyle={auto=right, fill=none}](v11)(v10)
\Edge[label = \tiny {}, labelstyle={auto=right, fill=none}](v11)(v12)
\Edge[label = \tiny {}, labelstyle={auto=right, fill=none}](v13)(v9)
\Edge[label = \tiny {}, labelstyle={auto=right, fill=none}](v13)(v12)
\Edge[label = \tiny {}, labelstyle={auto=right, fill=none}](v13)(v14)
\Edge[label = \tiny {}, labelstyle={auto=right, fill=none}](v15)(v8)
\Edge[label = \tiny {}, labelstyle={auto=right, fill=none}](v15)(v11)
\Edge[label = \tiny {}, labelstyle={auto=right, fill=none}](v15)(v14)
\Edge[label = \tiny {}, labelstyle={auto=right, fill=none}](v2)(v14)
\path 
(v6) edge [line width=1pt, bend right=45] (v16)
(v10) edge [line width=1pt, bend left=45] (v16);
\end{tikzpicture}
~~
\begin{tikzpicture}[scale = 6]
\tikzstyle{VertexStyle} = []
\tikzstyle{EdgeStyle} = []
\tikzstyle{labeledStyle}=[shape = circle, minimum size = 6pt, inner sep = 1.2pt, draw]
\tikzstyle{unlabeledStyle}=[shape = circle, minimum size = 6pt, inner sep = 1.2pt, draw, fill]
\Vertex[style = unlabeledStyle, x = 1, y = 0.80, L = \tiny {}]{v0}
\Vertex[style = unlabeledStyle, x = 0.90, y = 0.75, L = \tiny {}]{v1}
\Vertex[style = unlabeledStyle, x = 0.85, y = 0.65, L = \tiny {}]{v2}
\Vertex[style = unlabeledStyle, x = 0.90, y = 0.55, L = \tiny {}]{v3}
\Vertex[style = unlabeledStyle, x = 1.00, y = 0.50, L = \tiny {}]{v4}
\Vertex[style = unlabeledStyle, x = 1.10, y = 0.55, L = \tiny {}]{v5}
\Vertex[style = unlabeledStyle, x = 1.15, y = 0.65, L = \tiny {}]{v6}
\Vertex[style = unlabeledStyle, x = 1.10, y = 0.75, L = \tiny {}]{v7}
\Vertex[style = unlabeledStyle, x = 0.20, y = 0.80, L = \tiny {}]{v8}
\Vertex[style = unlabeledStyle, x = 0.10, y = 0.75, L = \tiny {}]{v9}
\Vertex[style = unlabeledStyle, x = 0.05, y = 0.65, L = \tiny {}]{v10}
\Vertex[style = unlabeledStyle, x = 0.10, y = 0.55, L = \tiny {}]{v11}
\Vertex[style = unlabeledStyle, x = 0.20, y = 0.50, L = \tiny {}]{v12}
\Vertex[style = unlabeledStyle, x = 0.30, y = 0.55, L = \tiny {}]{v13}
\Vertex[style = unlabeledStyle, x = 0.35, y = 0.65, L = \tiny {}]{v14}
\Vertex[style = unlabeledStyle, x = 0.30, y = 0.75, L = \tiny {}]{v15}
\Vertex[style = unlabeledStyle, x = 0.60, y = 0.95, L = \tiny {}]{v16}
\Vertex[style = unlabeledStyle, x = 0.60, y = 0.80, L = \tiny {}]{v17}
\Vertex[style = unlabeledStyle, x = 0.50, y = 0.75, L = \tiny {}]{v18}
\Vertex[style = unlabeledStyle, x = 0.45, y = 0.65, L = \tiny {}]{v19}
\Vertex[style = unlabeledStyle, x = 0.50, y = 0.55, L = \tiny {}]{v20}
\Vertex[style = unlabeledStyle, x = 0.60, y = 0.5, L = \tiny {}]{v21}
\Vertex[style = unlabeledStyle, x = 0.70, y = 0.55, L = \tiny {}]{v22}
\Vertex[style = unlabeledStyle, x = 0.75, y = 0.65, L = \tiny {}]{v23}
\Vertex[style = unlabeledStyle, x = 0.70, y = 0.75, L = \tiny {}]{v24}
\Edge[label = \tiny {}, labelstyle={auto=right, fill=none}](v1)(v0)
\Edge[label = \tiny {}, labelstyle={auto=right, fill=none}](v1)(v2)
\Edge[label = \tiny {}, labelstyle={auto=right, fill=none}](v3)(v2)
\Edge[label = \tiny {}, labelstyle={auto=right, fill=none}](v3)(v4)
\Edge[label = \tiny {}, labelstyle={auto=right, fill=none}](v5)(v1)
\Edge[label = \tiny {}, labelstyle={auto=right, fill=none}](v5)(v4)
\Edge[label = \tiny {}, labelstyle={auto=right, fill=none}](v5)(v6)
\Edge[label = \tiny {}, labelstyle={auto=right, fill=none}](v7)(v0)
\Edge[label = \tiny {}, labelstyle={auto=right, fill=none}](v7)(v3)
\Edge[label = \tiny {}, labelstyle={auto=right, fill=none}](v7)(v6)
\Edge[label = \tiny {}, labelstyle={auto=right, fill=none}](v9)(v8)
\Edge[label = \tiny {}, labelstyle={auto=right, fill=none}](v9)(v10)
\Edge[label = \tiny {}, labelstyle={auto=right, fill=none}](v11)(v10)
\Edge[label = \tiny {}, labelstyle={auto=right, fill=none}](v11)(v12)
\Edge[label = \tiny {}, labelstyle={auto=right, fill=none}](v13)(v9)
\Edge[label = \tiny {}, labelstyle={auto=right, fill=none}](v13)(v12)
\Edge[label = \tiny {}, labelstyle={auto=right, fill=none}](v13)(v14)
\Edge[label = \tiny {}, labelstyle={auto=right, fill=none}](v15)(v8)
\Edge[label = \tiny {}, labelstyle={auto=right, fill=none}](v15)(v11)
\Edge[label = \tiny {}, labelstyle={auto=right, fill=none}](v15)(v14)
\Edge[label = \tiny {}, labelstyle={auto=right, fill=none}](v18)(v17)
\Edge[label = \tiny {}, labelstyle={auto=right, fill=none}](v18)(v19)
\Edge[label = \tiny {}, labelstyle={auto=right, fill=none}](v20)(v19)
\Edge[label = \tiny {}, labelstyle={auto=right, fill=none}](v20)(v21)
\Edge[label = \tiny {}, labelstyle={auto=right, fill=none}](v22)(v18)
\Edge[label = \tiny {}, labelstyle={auto=right, fill=none}](v22)(v21)
\Edge[label = \tiny {}, labelstyle={auto=right, fill=none}](v22)(v23)
\Edge[label = \tiny {}, labelstyle={auto=right, fill=none}](v24)(v17)
\Edge[label = \tiny {}, labelstyle={auto=right, fill=none}](v24)(v20)
\Edge[label = \tiny {}, labelstyle={auto=right, fill=none}](v24)(v23)
\Edge[label = \tiny {}, labelstyle={auto=right, fill=none}](v19)(v14)
\Edge[label = \tiny {}, labelstyle={auto=right, fill=none}](v23)(v2)
\path [line width=1pt]
(v6) edge [bend right=55] (v16)
(v16) edge [bend right=55] (v10);
\end{tikzpicture}
\caption{
The first three examples in an infinite family of 3-critical graphs
with $2|E(G)|<(2+\frac34)|V(G)|$.
\label{woodall-examples}}
\end{center}
\end{figure}
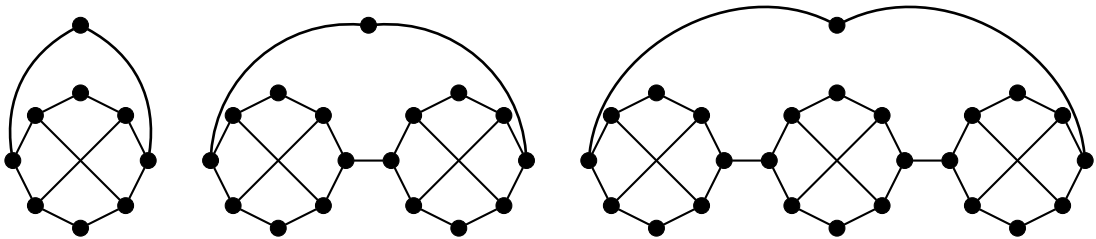
\bigskip

Our main result is the following.
\begin{thm}
\label{thm:main1}
Every 3-critical graph $G$ other than $P^*$ has average
degree $a(G)\ge \frac{46}{17}\approx 2.706$.
\end{thm}

By using a computer, we are able to improve our edge bound further.
\begin{thm}
Let $G$ be a 3-critical graph.  If $G$ is neither $P^*$ nor the
\hajos\ join of two copies of $P^*$, then $G$ has average degree $a(G)\ge
\frac{84}{31}\approx 2.710$.
\label{thm:main2}
\end{thm}
\noindent
However, a
human-readable proof is too long to include here (roughly 100 pages).  We
discuss this work a bit more in Section~\ref{computer}, as well as give
a web link where that proof is available.  


\section{Poor subgraphs, discharging, proofs of the theorems}
\label{sec:discharging}
In this section, we introduce some basic tools for proving edge-coloring
results.  We then prove our main result, subject to some reducibility lemmas
which we prove afterward.

Given a (partial) edge-coloring of a graph $G$, an \emph{$(x,y)$-Kempe chain}
(or simply \emph{$(x,y)$-chain})
is a component of the subgraph induced by the edges colored $x$ and $y$.  Note
that each Kempe chain is either a path or an even cycle.  
If vertices $v_1$ and $v_2$ lie in the same $(x,y)$-Kempe chain, then $v_1$ and
$v_2$ are \emph{$x,y$-linked}.
Given a proper coloring of (some subgraph of) a graph $G$, if we interchange
the colors on some $(x,y)$-chain, then the resulting coloring is again proper.
 This interchange is an \emph{$(x,y)$-Kempe swap} (or simply
\emph{$(x,y)$-swap}) and plays a central role in most proofs of forbidden
subgraphs in critical graphs.

Suppose that $d(v_1)=2$, $d(v_2)=3$, and $v_1v_2\in E(G)$.  Suppose also that we
3-color $G-v_1v_2$ with colors $x$, $y$, and $z$.  If we cannot extend this
coloring to $G$, then (by symmetry) we may assume that $v_1$ sees $x$ and that
$v_2$ sees $y$ and $z$.  Furthermore, $v_1$ and $v_2$ must be $x,y$-linked;
otherwise we perform an $(x,y)$-swap at $v_1$ and afterwards color $v_1v_2$
with $x$.  Similarly, $v_1$ and $v_2$ must be $x,z$-linked.

The quintessential tool for forbidding a subgraph in a critical graph is Vizing's
Adjacency Lemma, which he proved using Kempe chains and a similar structure
for recoloring, known as Vizing fans.
The proof is available in Fiorini \& Wilson~\cite[p. 72--74]{FioriniW77} and 
Stiebitz et al.~\cite[p. 19ff.]{StiebitzSTF12}. 

\begin{VAL}
Let $G$ be a $\Delta$-critical graph.  If $u,v\in V(G)$ and $uv\in E(G)$, then
the number of $\Delta$-neighbors of $u$ different from $v$ is at least $\Delta-d(v)+1$.
\end{VAL}
\noindent


We recall the following basic facts, which are well known.

\begin{lem}
For every 3-critical graph $G$, the following hold:
\begin{enumerate}
\item[(a)] $G$ is 2-connected; and
\item[(b)] $G$ has no adjacent 2-vertices, and $G$ has no 3-vertex with two or more 2-neighbors.
\label{lem:basic}
\end{enumerate}
\end{lem}

\begin{proof}
Note that (b) follows immediately from Vizing's Adjacency Lemma.
Hence, it suffices to prove (a).  If $G$ is disconnected, then we can
3-color each component by criticality, a
contradiction.  Similarly, suppose that $G$ has a cut-edge $v_1v_2$ and let
$G_1$ and $G_2$ be the components of $G-v_1v_2$.  By criticality, we can 3-color
both $G_1$ and $G_2$.  Further, we can permute the colors on $G_2$ so that $v_1$
and $v_2$ miss a common color, say $x$.  Now we color $v_1v_2$ with $x$ to get a
3-coloring of $G$, a contradiction.  Thus $G$ is 2-edge-connected. For a
subcubic graph this is the same as 2-connected.
%
%
\end{proof}
\smallskip

Let $G$ be a 3-critical graph.  Lemma~\ref{lem:basic}(a) implies that every vertex of
$G$ has degree 2 or 3.  We call a 3-vertex \emph{poor} if it has a 2-neighbor
and \emph{rich} otherwise.  The \emph{poor subgraph} of $G$ is the subgraph $H$
induced by the poor vertices; its components, denoted by $H_i$, are the
\emph{poor fragments} of $G$.  Every poor vertex has exactly one 2-neighbor, so
at most two poor neighbors.  Thus, every poor fragment has maximum degree at
most 2, so is a path or a cycle.

For a rich vertex $w$, let $p(w)$ denote the sum of the orders of all the poor
fragments in which $w$ has neighbors, accounting for multiplicity (so if $w$ is
adjacent to both end-vertices of a poor fragment $H_i$, then $|V(H_i)|$ counts
twice towards $p(w)$).  Let $p(G)=\max\{p(w)$ : $w$ is a rich vertex of $G$\}. 
The next result is the foundation of our proof.

\begin{lem}
Let $G$ be a 3-critical graph such that every poor fragment is a path (not a
cycle) and let $p=p(G)$.  Now $G$ has average degree $a(G)\ge
\frac{8p+12}{3p+4}$.
\label{keylemma}
\end{lem}
\begin{proof}
We use the discharging method with initial charge $\ch(v)=d(v)$ for each vertex
$v$ and the following three discharging rules.
\begin{enumerate}
\item[(R1)] Each 2-vertex takes charge $\frac{p+2}{3p+4}$ from each neighbor.
\item[(R2)] Each rich vertex gives charge $\frac{t}{3p+4}$ to each neighbor in a
poor fragment of order $t$.
\item[(R3)] All vertices within a poor fragment share charge equally.
\end{enumerate}

Now we show that each vertex $v$ finishes with final charge $\ch^*(v)\ge 
\frac{8p+12}{3p+4}$, which proves the lemma.  If $v$ is a 2-vertex, then
$\ch^*(v)=2+2(\frac{p+2}{3p+4}) = \frac{8p+12}{3p+4}$.  If $v$ is a rich vertex,
then $\ch^*(v)\ge 3-\frac{p}{3p+4} = \frac{8p+12}{3p+4}$.  Now suppose that $v$
is a poor vertex in a poor fragment $H_i$, and let $t=|V(H_i)|$.  Since $H_i$ is a
path, each of its two end-vertices must be adjacent to a rich vertex.  Thus,
these end-vertices receive in total $2(\frac{t}{3p+4})$.  So by (R3), each vertex
of $H_i$ receives $\frac2{3p+4}$.  Thus $\ch^*(v) = 3 - \frac{p+2}{3p+4}
+\frac2{3p+4} = \frac{8p+12}{3p+4}$, as desired.
\end{proof}

If we let $f(p)=\frac{8p+12}{3p+4}$, then $f(10)=\frac{46}{17}\approx 2.706$ and
$f(9)=\frac{84}{31}\approx 2.710$.  So to prove Theorems 1 and 2 it suffices to
prove that (for appropriate graphs $G$) always $p(G)\le 10$ and $p(G)\le 9$,
respectively.  Subject to some lemmas that we prove below and in the following
sections, we can now prove the two main theorems. For convenience, we restate them.

\begin{mainthm1}
Every 3-critical graph $G$ other than $P^*$ has average
degree $a(G)\ge \frac{46}{17}\approx 2.706$.
\end{mainthm1}

\begin{proof}
Suppose the theorem is false, and let $G$ be a minimal counterexample.
We will use Lemma~\ref{keylemma}.  So it suffices to show that (i) every poor
fragment is a path and (ii) $p(G)\le 10$.
Since $\frac{46}{17}<\frac{30}{11}$, Lemma~\ref{lem:summary} shows that every
poor fragment is a path with at most 5 vertices.  Thus, it suffice to show that
$p(G)\le 10$.  

Consider an arbitrary rich vertex $w$.  If every poor fragment adjacent to $w$
has order at most 3, then $p(w)\le 3(3)=9$.  So suppose $w$ has some adjacent
poor fragment $H_1$ with order at least 4; recall from above that $H_1$ has
order at most 5.  By the last sentence of Lemma~\ref{lem:summary}, 
each poor neighbor of $w$ is in a distinct poor fragment.
By Lemma~\ref{two-jellyfish}, each adjacent poor fragment other than $H_1$ has
order at most 3.  Thus $p(w)\le 5+3+3=11$.  Further, equality
holds only if $w$ is adjacent to poor fragments of orders 5, 3, and 3.  By
Lemma~\ref{lem:new3}, this is impossible.  Thus, $p(w)\le 10$, as desired.
\end{proof}

\begin{mainthm2}
Let $G$ be a 3-critical graph.  If $G$ is neither $P^*$ nor the
\hajos\ join of two copies of $P^*$, then $G$ has average degree $a(G)\ge
\frac{84}{31}\approx 2.710$.
\end{mainthm2}

\begin{proof}
The proof is nearly the same as that of Theorem~\ref{thm:main1}.  Now, in
addition, we use a computer to show that if $G$ is not the \hajos\ join of two
copies of $P^*$, then no rich vertex is adjacent to two poor
fragments, each of order at least 3.  Thus, for each rich vertex $w$, the
largest order of an adjacent poor fragment is at most 5; if this order is at
least 4, then each of the remaining adjacent poor fragments has
order at most 2.  Hence, $p(w)\le 5+2+2=9$ or $p(w)\le 3+3+3=9$;
so, $p(G)\le 9$. 
\end{proof}

We summarize the remaining results of this section in the following lemma. For
the proofs
of Theorems~\ref{thm:main1} and~\ref{thm:main2}, we need to know that this result
holds when $\alpha=\frac{46}{17}\approx 2.706$ and $\alpha=\frac{84}{31}\approx
2.710$ respectively, which it does, since $\frac83\approx
2.667<\frac{46}{17}<\frac{84}{31}<2.75=\frac{11}4$.  If $H_i$ is a poor fragment
of $G$, then let $H_i^+$ denote the subgraph of $G$ induced by the vertices of
$H_i$ and the neighboring 2-vertices.

\begin{lem} 
\label{lem:summary}
Let $\frac83\le \alpha \le \frac{11}4$, let $G$ be a 3-critical
graph such that $a(G)<\alpha$ and $G$ is neither $P^*$ nor the (successive)
\hajos\ join of multiple copies of $P^*$ (as in
Example~\ref{woodall-examples}); among all 3-critical graphs with these
properties, choose $G$ to have the fewest vertices.  Assume that either
$G$ is triangle-free or $\alpha\le \frac{30}{11}=2.\overline{72}$.  Now every
poor fragment $H_i$ of $G$ is a path on at most 5 vertices, and $H_i^+$ has one
of the forms shown in Figure~\ref{fig:H_i^+}.  In addition, no rich vertex is
adjacent to both end-vertices of $H_i$.
\end{lem}

We prove Lemma~\ref{lem:summary} in a sequence of six claims.

\begin{clm}
Suppose $0<n<|V(G)|$ and $J$ is a connected subcubic graph with $|V(G)|-n$
vertices and $|E(G)|-m$ edges.  Suppose that one of the following holds: (i)
$\frac{m}n\ge \frac{\alpha}2$ and $J$ is not a \hajos\ join of copies of $P^*$
or (ii) $\frac{12+m}{9+n}\ge \frac{\alpha}2$.  Now $J$ is not 3-critical.
\end{clm}

\begin{proof}
Suppose first that (i) holds.  Since $\frac{2m}n\ge
\alpha > a(G)=\frac{2|E(G)|}{|V(G)|}$, it follows that
$a(J)=\frac{2|E(G)|-2m}{|V(G)|-n}<\frac{2|E(G)|}{|V(G)|}=a(G)<\alpha$.  Since
$J$ is not a \hajos\ join of copies of $P^*$ by (i), it follows from the
minimality of $G$ that $J$ is not 3-critical.

We now prove that (ii) implies (i).  Since $\frac{12}9=\frac43\le
\frac{\alpha}2$, the hypothesis implies that $\frac{m}n\ge \frac{\alpha}2$.
Suppose now that $J$ is a \hajos\ join of $k$ copies of $P^*$, so that
$|V(J)|=8k+1$ and $|E(J)|=11k+1$.  Recall that $\frac{2(12+m)}{9+n}\ge \alpha$ by
(ii), and $\frac{11}4\ge \alpha$.  Thus, $a(G)= \frac{2(|E(J)|+m)}{|V(J)|+n} =
\frac{2(11k+1+m)}{8k+1+n} = \frac{2(11(k-1)+(12+m))}{8(k-1)+(9+n)}\ge \alpha$. 
This contradiction shows that $J$ is not a \hajos\ join of copies of $P^*$, and
so (i) holds. 
\end{proof}

We now prove some structural properties of $G$.

\begin{clm}
$G$ does not have a subgraph $F$ containing exactly two vertices, $v_1$ and
$v_2$, with neighbors outside $F$ such that $F+v_1v_2\cong P^*$.
\label{clm:noPete}
\end{clm}

\begin{proof}
Suppose it does.  Since every edge of $P^*$ has at least one end-vertex of
degree 3, we may assume that $v_2$ has degree 3\ in $P^*$ and degree 2\ in $F$.
Since $G$ is subcubic and 2-connected by Lemma~\ref{lem:basic}, there is exactly one
edge $v_2v_4$ in $G$ with $v_4\notin F$.  It is now easy to see that $G$ is the
\hajos\ join of $P^*$ and a connected subcubic graph $J$ smaller than $G$.  Note
that $J$ is not a \hajos\ join of copies of $P^*$, since otherwise $G$ would be as
well, which it is not.

By Corollary~\ref{hajos-cor}, $J$ is 3-critical.  Now since $|V(J)|=|V(G)|-8$ and
$|E(J)|=|E(G)|-11$ and $\frac{11}8\ge \frac{\alpha}2$, it follows from
Claim~1(i) that $J$ is not 3-critical. This is the required contradiction.
\end{proof}

\begin{clm}
If $\alpha\le \frac{30}{11}=2.\overline{72}$, then $G$ is triangle-free.
\label{clm1}
\end{clm}
\begin{proof}
Suppose that $G$ has a 3-cycle $v_1v_2v_3$.  Suppose first that no edge of
$v_1v_2v_3$ lies in another triangle.  Form $G'$ from $G$ by contracting the
three edges of triangle $v_1v_2v_3$.  Since $|V(G')|=|V(G)|-2$ and
$|E(G')|=|E(G)|-3$ and $\frac{12+3}{9+2}\ge \frac{\alpha}{2}$, Claim~1(ii)
implies that $G'$ is not 3-critical.  But it is easy to see that if any two
subcubic graphs $G$ and $G'$ are related in this way (with a vertex of $G'$
corresponding to a triangle in $G$), then $G'$ is 3-critical if and only if $G$
is, so we have a contradiction.

Thus we may assume that some edge of $v_1v_2v_3$ lies in another triangle, say
$v_1v_2v_4$.  Note that $v_3$ and $v_4$ are not adjacent and have no common
neighbor of degree 2, since this would imply $(|V(G)|,|E(G)|)$ is $(4,6)$ or
$(5,7)$ and $a(G)\ge \frac{14}5=2.8>\alpha$.  Since $G$ is 2-connected, $v_3$
and $v_4$ have distinct neighbors $v_5$ and $v_6$, respectively.  

A \emph{ladder} is any graph that is formed from the disjoint union of paths
$w_1\cdots w_\ell$ and $z_1\cdots z_\ell$ by adding all edges $z_iw_i$ for $2
\le i \le \ell-1$ as well as possibly adding $w_1z_1$ or $w_\ell z_\ell$ (or
both); such edges $z_iw_i$ are \emph{rungs}.  Let $L$ be a maximal induced
subgraph of $G$ containing $v_1, \ldots, v_6$ such that $L - \set{v_1, v_2}$ is
a ladder.  Let $x_1, x_2$ be the vertices on the opposite end of the ladder
from $v_3, v_4$.

First, suppose $x_1$ and $x_2$ are nonadjacent.  Form $L'$ from $L$ by
identifying $x_1$ and $x_2$ to create a new vertex $x'$.  Note that $G$ is a
\hajos\ join of $L'$ and another graph $J$.  Suppose the ladder $L - \set{v_1,
v_2}$ has $t$ rungs.  Then $|V(J)| = |V(G)| - 2(t+2)$ and $|E(J)| = |E(G)| -
3(t+2)$.  Since $\frac{12 + 3(t+2)}{9 + 2(t+2)} \ge \frac{\alpha}{2}$ for all
$t \ge 0$, Claim~1(ii) shows that $J$ is not 3-critical.  Since $G$ is
3-critical, Corollary~\ref{hajos-cor} shows that $L'$ is not 3-critical.  It is
easy to see that $L'$ is 4-chromatic, since it is 3-regular, except for the
single 2-vertex $x'$.  So, $L'$ has a $3$-critical proper subgraph $Q$.  
Now $x' \in V(Q)$ since $L'-x'$ is a proper subgraph of $G$.  
By Lemma \ref{lem:basic}(a,b), $Q$ contains every edge in $G[v_1,v_2,v_3,v_4]$.  
By Lemma \ref{lem:basic}(a), $Q$ contains all edges of the ladder that are not rungs.  
If $Q$ is missing a rung of the ladder, then $Q$ contains an edge-cut
$\set{e_1, e_2}$ with each $e_i$ incident to a
2-vertex, which is impossible since $Q$ is $3$-critical.  Hence $Q = L'$, a
contradiction.

So, $x_1$ and $x_2$ must be adjacent.  Now $x_1$ and $x_2$ cannot have distinct
neighbors outside $L$, by the maximality of $L$.  But $x_1$ and $x_2$ have
no common neighbor either since then $G$ would have either a cut-edge or a
triangle with no edge in two triangles (contradicting minimality, by the first
paragraph of this proof of Claim~3).  To avoid a cut-edge, the only
remaining possibility is that $x_1$ and $x_2$ are both 2-vertices.  But now
$G$ is easily $3$-colored, a contradiction.
\end{proof}

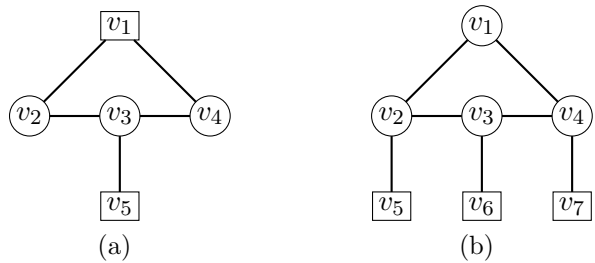
\begin{figure}[hbt]
\begin{center}

~~~
\subfloat[]{
\makebox[.25\textwidth]{
\begin{tikzpicture}[scale = 12]
\tikzstyle{VertexStyle} = []
\tikzstyle{EdgeStyle} = []
\tikzstyle{unlabeledStyle}=[shape = circle, minimum size = 6pt, inner sep = 1.2pt, draw, fill]
\Vertex[style = labeledTwo, x = 0.60, y = 0.80, L = \small {$v_1$}]{v0}
\Vertex[style = labeledThree, x = 0.50, y = 0.70, L = \small {$v_2$}]{v1}
\Vertex[style = labeledThree, x = 0.60, y = 0.70, L = \small {$v_3$}]{v2}
\Vertex[style = labeledThree, x = 0.70, y = 0.70, L = \small {$v_4$}]{v3}
\Vertex[style = labeledTwo, x = 0.60, y = 0.60, L = \small {$v_5$}]{v4}
\Edge[](v3)(v2)
\Edge[](v3)(v0)
\Edge[](v1)(v0)
\Edge[](v1)(v2)
\Edge[](v2)(v4)
\end{tikzpicture}
}}
~~~
\subfloat[]{
\makebox[.25\textwidth]{
\begin{tikzpicture}[scale = 12]
\tikzstyle{VertexStyle} = []
\tikzstyle{EdgeStyle} = []
\tikzstyle{unlabeledStyle}=[shape = circle, minimum size = 6pt, inner sep =
1.2pt, draw, fill]
\Vertex[style = labeledThree, x = 0.60, y = 0.80, L = \small {$v_1$}]{v0}
\Vertex[style = labeledThree, x = 0.50, y = 0.70, L = \small {$v_2$}]{v1}
\Vertex[style = labeledThree, x = 0.60, y = 0.70, L = \small {$v_3$}]{v2}
\Vertex[style = labeledThree, x = 0.70, y = 0.70, L = \small {$v_4$}]{v3}
\Vertex[style = labeledTwo, x = 0.70, y = 0.60, L = \small {$v_7$}]{v4}
\Vertex[style = labeledTwo, x = 0.60, y = 0.60, L = \small {$v_6$}]{v5}
\Vertex[style = labeledTwo, x = 0.50, y = 0.60, L = \small {$v_5$}]{v6}
\Edge[](v3)(v4)
\Edge[](v3)(v2)
\Edge[](v3)(v0)
\Edge[](v1)(v6)
\Edge[](v1)(v0)
\Edge[](v1)(v2)
\Edge[](v2)(v5)
\end{tikzpicture}
}}
\end{center}
\caption{Two subgraphs forbidden from a 3-critical graph $G$. 
Vertices drawn as rectangles have degree 2 in $G$ and those drawn as circles
have degree 3 in $G$.}
\label{umbrella-pic}
\label{jellyfish-pic}
\end{figure}

\begin{clm}
Neither of the configurations in Figures~\ref{umbrella-pic}(a) and
\ref{jellyfish-pic}(b) is a subgraph of $G$.
\label{clm2}
\end{clm}
\begin{proof}
We begin with Figure~\ref{umbrella-pic}(a).  
Since $G$ is 3-critical, $\chi'(G)=4$, but there is a 3-coloring $\varphi$ of
$G-v_3v_5$ using colors $x$, $y$, and $z$.  Since $v_3v_5$ cannot be colored,
we may assume that $v_5$ sees color $x$, $\varphi(v_2v_3)=y$, and
$\varphi(v_3v_4)=z$.  First suppose that $v_1$ sees $x$.  By symmetry (in the
graph and also between colors $y$ and $z$), assume $\varphi(v_1v_2)=x$ and
$\varphi(v_1v_4)=y$. Recolor $v_2v_3$ with $x$, recolor $v_1v_2$ and $v_3v_4$ with
$y$, and recolor $v_1v_4$ with $z$.  Now we can color $v_3v_5$ with $z$.  So
assume instead that $v_1$ misses $x$; thus $\varphi(v_1v_2)=z$ and
$\varphi(v_1v_4)=y$.  Now do an $(x,y)$-swap at $v_5$.  Edge $v_3v_5$
will be colorable with $x$ unless the $(x,y)$-path starting at $v_5$ ends
at $v_3$; so assume that it does.  Now $v_5$ sees $y$ and $v_1$ sees $y$, so we
can extend the coloring to $G$ as above (with $y$ in the role of $x$).

Now consider Figure~\ref{jellyfish-pic}(b).  As above, we use colors $x$, $y$, $z$
to 3-color $G-v_3v_6$; call the coloring $\varphi$.  
As before, we assume that $v_6$ sees color $x$, $\varphi(v_2v_3)=y$, and
$\varphi(v_3v_4)=z$.  We assume by symmetry that $\varphi(v_1v_2)\ne x$, so
that $\varphi(v_1v_2)=z$ and $\varphi(v_2v_5)=x$.
We may
assume that $v_3$ and $v_6$ are $x,y$-linked.  Thus, $v_5$ sees $y$.  If
$\varphi(v_1v_4)=y$, then do a $(y,z)$-swap at $v_3$ (the entire component is
just the 4-cycle $v_1v_2v_3v_4$).  Now $v_3$ and $v_6$
are no longer $x,z$-linked, so do an $(x,z)$-swap at $v_3$, and color
$v_3v_6$ with $z$.  Thus, we assume that $\varphi(v_1v_4)=x$.  Now again, do an
$(x,z)$-swap at $v_3$, then color $v_3v_6$ with $z$.
\end{proof}

\begin{clm}
Every poor fragment $H_i$ of $G$ is a path on at most 5 vertices, and $H_i^+$
has one of the forms in Figure~\ref{fig:H_i^+}.  
\label{clm3}
\end{clm}
\begin{proof}
Suppose not.  By construction, $\Delta(H_i)\le 2$; since $G$ has no 3-cycles
(by Claim~\ref{clm1} and the hypothesis of Lemma~\ref{lem:summary}), assume
that some poor fragment $H_1$ induces a path or
cycle $v_1\cdots v_t$ on four or more vertices.  
Since $G$ has no 3-cycles, no successive 3-vertices on $H_1$ have a common
2-neighbor.  Similarly, since the configuration in Figure~\ref{umbrella-pic}(a)
is forbidden, no vertices at distance two on $H_1$ have a common 2-neighbor.
By Lemma~\ref{lem:fig2a}, no four consecutive vertices on $H_1$ have distinct
2-neighbors.  Thus, each vertex $v_i\in V(H_1)$ (for $i \in \{1,\ldots,t-3\}$)
must share a common 2-neighbor with $v_{i+3}$. This immediately gives that $t
\le 6$, since otherwise $v_4$ must share a common 2-neighbor with both $v_1$
and $v_7$, a contradiction.

If $H_1$ is a 6-cycle, then (since $G$ is subcubic and connected), $G=H_1^+$,
and also $H_1^+\cong P^*$; this contradicts the definition of $G$ in
Lemma~\ref{lem:summary}.  (Recall that $P^*$ can alternatively be drawn with a
convex 6-cycle $C$, where each pair of vertices at distance 3 on $C$ have a commong
2-neighbor.) If $H_1$ is a path of order 6, then $v_1$ and $v_6$
are the only vertices of $H_1^+$ that can have neighbors outside $H_1^+$, so
$H_1^++v_1v_6\cong P^*$; this contradicts Claim~\ref{clm:noPete}.
If $H_1$ is a 5-cycle, then $v_1$ and $v_4$ have a common 2-neighbor; similarly,
$v_2$ and $v_5$ have a common 2-neighbor.  Thus, the edge from $v_3$ to its
2-neighbor is a cut-edge, contradicting Lemma~\ref{lem:basic}(a).  If $H_1$ is a
4-cycle, then $v_1$ and $v_4$ have a common 2-neighbor, but they are also
adjacent.  This 3-cycle contradicts Claim~\ref{clm1}.
Thus, $H_1$ is a path on at most five vertices.  Further, since the
configuration in Figure~\ref{fig:2a} is forbidden by Lemma~\ref{lem:fig2a},
$H_i^+$ has one of the forms in Figure~\ref{fig:H_i^+}.
\end{proof}

\begin{figure}
\centering
\subfloat[]{
\makebox[.08\textwidth]{
\begin{tikzpicture}[scale = 12]
\tikzstyle{VertexStyle} = []
\tikzstyle{EdgeStyle} = []
\tikzstyle{labeledThree}=[shape = circle, minimum size = 8pt, inner sep = 2.2pt, draw]
\tikzstyle{labeledTwo}=[shape = rectangle, minimum size = 8pt, inner sep = 3.2pt, draw]
\tikzstyle{labeledStyle}=[shape = circle, minimum size = 6pt, inner sep = 1.2pt, draw]
\tikzstyle{unlabeledStyle}=[shape = circle, minimum size = 6pt, inner sep = 1.2pt, draw, fill]
\Vertex[style = labeledTwo, x = 0.450, y = 0.750, L = \large {$\>$}]{v0}
\Vertex[style = labeledThree, x = 0.450, y = 0.650, L = \large {$~$}]{v1}
\Edge[label = \tiny {}, labelstyle={auto=right, fill=none}](v1)(v0)
\end{tikzpicture}
}}
\subfloat[]{
\makebox[.17\textwidth]{
\begin{tikzpicture}[scale = 12]
\tikzstyle{VertexStyle} = []
\tikzstyle{EdgeStyle} = []
\tikzstyle{labeledThree}=[shape = circle, minimum size = 8pt, inner sep = 2.2pt, draw]
\tikzstyle{labeledTwo}=[shape = rectangle, minimum size = 8pt, inner sep = 3.2pt, draw]
\tikzstyle{labeledStyle}=[shape = circle, minimum size = 6pt, inner sep = 1.2pt, draw]
\tikzstyle{unlabeledStyle}=[shape = circle, minimum size = 6pt, inner sep = 1.2pt, draw, fill]
\Vertex[style = labeledTwo, x = 0.450, y = 0.750, L = \large {$\>$}]{v0}
\Vertex[style = labeledThree, x = 0.450, y = 0.650, L = \large {$~$}]{v1}
\Vertex[style = labeledThree, x = 0.550, y = 0.650, L = \large {$~$}]{v2}
\Vertex[style = labeledTwo, x = 0.550, y = 0.750, L = \large {$\>$}]{v3}
\Edge[label = \tiny {}, labelstyle={auto=right, fill=none}](v2)(v3)
\Edge[label = \tiny {}, labelstyle={auto=right, fill=none}](v2)(v1)
\Edge[label = \tiny {}, labelstyle={auto=right, fill=none}](v1)(v0)
\end{tikzpicture}
}}
\subfloat[]{
\makebox[.20\textwidth]{
\begin{tikzpicture}[scale = 12]
\tikzstyle{VertexStyle} = []
\tikzstyle{EdgeStyle} = []
\tikzstyle{labeledThree}=[shape = circle, minimum size = 8pt, inner sep = 2.2pt, draw]
\tikzstyle{labeledTwo}=[shape = rectangle, minimum size = 8pt, inner sep = 3.2pt, draw]
\tikzstyle{labeledStyle}=[shape = circle, minimum size = 6pt, inner sep = 1.2pt, draw]
\tikzstyle{unlabeledStyle}=[shape = circle, minimum size = 6pt, inner sep = 1.2pt, draw, fill]
\Vertex[style = labeledTwo, x = 0.450, y = 0.750, L = \large {$\>$}]{v0}
\Vertex[style = labeledThree, x = 0.450, y = 0.650, L = \large {$~$}]{v1}
\Vertex[style = labeledThree, x = 0.550, y = 0.650, L = \large {$~$}]{v2}
\Vertex[style = labeledTwo, x = 0.550, y = 0.750, L = \large {$\>$}]{v3}
\Vertex[style = labeledTwo, x = 0.650, y = 0.750, L = \large {$\>$}]{v4}
\Vertex[style = labeledThree, x = 0.650, y = 0.650, L = \large {$~$}]{v5}
\Edge[label = \tiny {}, labelstyle={auto=right, fill=none}](v5)(v4)
\Edge[label = \tiny {}, labelstyle={auto=right, fill=none}](v5)(v2)
\Edge[label = \tiny {}, labelstyle={auto=right, fill=none}](v2)(v3)
\Edge[label = \tiny {}, labelstyle={auto=right, fill=none}](v2)(v1)
\Edge[label = \tiny {}, labelstyle={auto=right, fill=none}](v1)(v0)
\end{tikzpicture}
}}
\subfloat[]{
\makebox[.23\textwidth]{
\begin{tikzpicture}[scale = 10]
\tikzstyle{VertexStyle} = []
\tikzstyle{EdgeStyle} = []
\tikzstyle{labeledThree}=[shape = circle, minimum size = 8pt, inner sep = 2.2pt, draw]
\tikzstyle{labeledTwo}=[shape = rectangle, minimum size = 8pt, inner sep = 3.2pt, draw]
\tikzstyle{unlabeledStyle}=[shape = circle, minimum size = 6pt, inner sep = 1.2pt, draw, fill]
\Vertex[style = labeledThree, x = 0.30, y = 0.80, L = \Large {$~$}]{v9}
\Vertex[style = labeledThree, x = 0.40, y = 0.80, L = \Large {$~$}]{v10}
\Vertex[style = labeledThree, x = 0.5, y = 0.80, L = \Large {$~$}]{v11}
\Vertex[style = labeledTwo, x = 0.35, y = 0.95, L = \Large {$\>$}]{v12}
\Vertex[style = labeledThree, x = 0.20, y = 0.80, L = \Large {$~$}]{v13}
\Vertex[style = labeledTwo, x = 0.30, y = 0.70, L = \Large {$\>$}]{v14}
\Vertex[style = labeledTwo, x = 0.40, y = 0.70, L = \Large {$\>$}]{v16}
\Edge[label = \small {}, labelstyle={auto=right, fill=none}](v9)(v13)
\Edge[label = \small {}, labelstyle={auto=right, fill=none}](v9)(v14)
\Edge[label = \small {}, labelstyle={auto=right, fill=none}](v10)(v9)
\Edge[label = \small {}, labelstyle={auto=right, fill=none}](v10)(v16)
\Edge[label = \small {}, labelstyle={auto=right, fill=none}](v11)(v10)
\Edge[label = \small {}, labelstyle={auto=right, fill=none}](v11)(v12)
\Edge[label = \small {}, labelstyle={auto=right, fill=none}](v12)(v13)
\end{tikzpicture}
}}
\subfloat[]{
\makebox[.30\textwidth]{
\begin{tikzpicture}[scale = 10]
\tikzstyle{VertexStyle} = []
\tikzstyle{EdgeStyle} = []
\tikzstyle{labeledThree}=[shape = circle, minimum size = 8pt, inner sep = 2.2pt, draw]
\tikzstyle{labeledTwo}=[shape = rectangle, minimum size = 8pt, inner sep = 3.2pt, draw]
\tikzstyle{labeledStyle}=[shape = circle, minimum size = 6pt, inner sep = 1.2pt, draw]
\tikzstyle{unlabeledStyle}=[shape = circle, minimum size = 6pt, inner sep = 1.2pt, draw, fill]
\Vertex[style = labeledThree, x = 0.550, y = 0.700, L = \large {$~$}]{v0}
\Vertex[style = labeledThree, x = 0.450, y = 0.700, L = \large {$~$}]{v1}
\Vertex[style = labeledThree, x = 0.750, y = 0.700, L = \large {$~$}]{v2}
\Vertex[style = labeledTwo, x = 0.600, y = 0.850, L = \large {$~$}]{v3}
\Vertex[style = labeledThree, x = 0.650, y = 0.700, L = \large {$~$}]{v4}
\Vertex[style = labeledTwo, x = 0.650, y = 0.600, L = \small
{$~$}]{v5}
\Vertex[style = labeledThree, x = 0.850, y = 0.700, L = \large {$~$}]{v8}
\Vertex[style = labeledTwo, x = 0.700, y = 0.850, L = \large {$~$}]{v9}
\Edge[label = \small {}, labelstyle={auto=right, fill=none}](v1)(v0)
\Edge[label = \small {}, labelstyle={auto=right, fill=none}](v2)(v3)
\Edge[label = \small {}, labelstyle={auto=left, fill=none}](v2)(v4)
\Edge[label = \small {}, labelstyle={auto=right, fill=none}](v3)(v1)
\Edge[label = \small {}, labelstyle={auto=left, fill=none}](v4)(v0)
\Edge[label = \small {}, labelstyle={auto=right, fill=none}](v5)(v4)
\Edge[label = \small {}, labelstyle={auto=left, fill=none}](v8)(v2)
\Edge[label = \small {}, labelstyle={auto=right, fill=none}](v9)(v0)
\Edge[label = \small {}, labelstyle={auto=left, fill=none}](v9)(v8)
\end{tikzpicture}
}}
\caption{The five possibilities for $H_i^+$.\label{fig:H_i^+}
}
\end{figure}
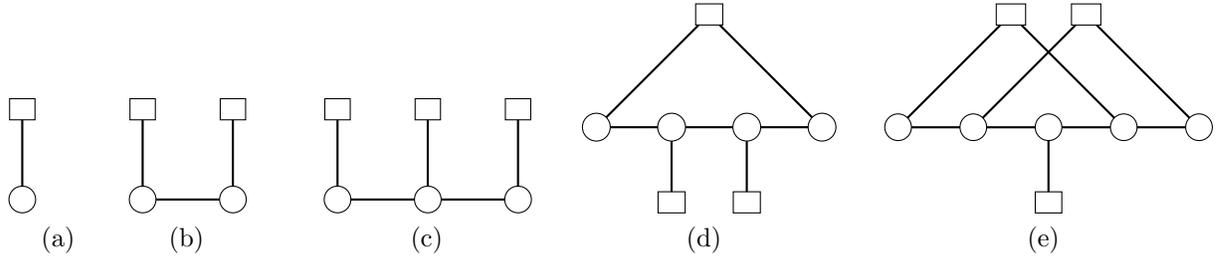

\begin{clm}
No rich 3-vertex has two neighbors in the same poor fragment.
\label{clm4}
\end{clm}
\begin{proof}
Suppose that $G$ contains such a 3-vertex $v$, and let $H_1$ be the poor
fragment containing two neighbors.  Claim~\ref{clm3} implies that
$H_1$ is a path on at most 5 vertices; further, $v$ must be adjacent
to the endvertices of $H_1$.  
If $H_1$ has order 2, then $G$ contains a 3-cycle, contradicting Claim~\ref{clm1}.
If $H_1$ has order 3, then $G$ contains Figure~\ref{jellyfish-pic}(b),
contradicting Claim~\ref{clm2}.
If $H_1$ has order 4, then we can color $G-H_1$ by criticality, and extend
the coloring to $G$
via one of the two extensions shown in Figure~\ref{order4-pic}, depending on
which colors are available at the 2-vertices. 
(By symmetry, assume that $v$ sees color $x$.  If $x$ is also seen by both
2-vertices in $H_1$ with edges to vertices outside $H_1$, then we use the
extension on the left; otherwise, the extension on the right.)
Finally, suppose that $H_1$ has order 5.  Now $G$ is the \hajos\ join of
$P^*$ and a smaller graph $J$; the copy of $P^*-e$ in $G$ consists
of $H_1^+$, $v$, and $v$'s neighbor outside of $H_1$.  
Corollary~\ref{hajos-cor} implies that $J$ is 3-critical.  Now $a(J)<a(G)$,
which contradicts our choice of $G$ as a minimal counterexample.
\end{proof}

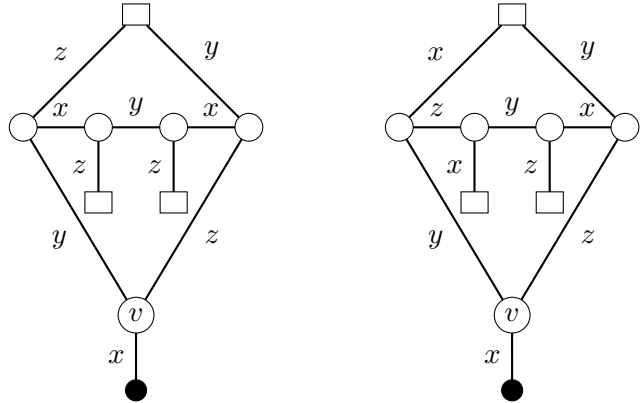
\begin{figure}[hbt]
\begin{center}
\begin{tikzpicture}[scale = 10]
\tikzstyle{VertexStyle} = []
\tikzstyle{EdgeStyle} = []
\tikzstyle{labeledThree}=[shape = circle, minimum size = 8pt, inner sep = 2.2pt, draw]
\tikzstyle{labeledTwo}=[shape = rectangle, minimum size = 8pt, inner sep = 3.2pt, draw]
\tikzstyle{unlabeledStyle}=[shape = circle, minimum size = 6pt, inner sep = 1.2pt, draw, fill]
\Vertex[style = labeledThree, x = 0.80, y = 0.80, L = \Large {$~$}]{v0}
\Vertex[style = labeledThree, x = 0.90, y = 0.80, L = \Large {$~$}]{v1}
\Vertex[style = labeledThree, x = 1, y = 0.80, L = \Large {$~$}]{v2}
\Vertex[style = labeledTwo, x = 0.85, y = 0.95, L = \Large {$\>$}]{v3}
\Vertex[style = labeledThree, x = 0.70, y = 0.80, L = \Large {$~$}]{v4}
\Vertex[style = labeledTwo, x = 0.80, y = 0.70, L = \Large {$\>$}]{v5}
\Vertex[style = labeledThree, x = 0.85, y = 0.55, L = \small {$v$}]{v6}
\Vertex[style = labeledTwo, x = 0.90, y = 0.70, L = \Large {$\>$}]{v7}
\Vertex[style = unlabeledStyle, x = 0.85, y = 0.45, L = \Large {$~$}]{v8}
\Vertex[style = labeledThree, x = 0.30, y = 0.80, L = \Large {$~$}]{v9}
\Vertex[style = labeledThree, x = 0.40, y = 0.80, L = \Large {$~$}]{v10}
\Vertex[style = labeledThree, x = 0.5, y = 0.80, L = \Large {$~$}]{v11}
\Vertex[style = labeledTwo, x = 0.35, y = 0.95, L = \Large {$\>$}]{v12}
\Vertex[style = labeledThree, x = 0.20, y = 0.80, L = \Large {$~$}]{v13}
\Vertex[style = labeledTwo, x = 0.30, y = 0.70, L = \Large {$\>$}]{v14}
\Vertex[style = labeledThree, x = 0.35, y = 0.55, L = \small {$v$}]{v15}
\Vertex[style = labeledTwo, x = 0.40, y = 0.70, L = \Large {$\>$}]{v16}
\Vertex[style = unlabeledStyle, x = 0.35, y = 0.45, L = \Large {$~$}]{v17}
\Edge[label = \small {$y$}, labelstyle={auto=right, fill=none}](v1)(v0)
\Edge[label = \small {$x$}, labelstyle={auto=right, fill=none}](v2)(v1)
\Edge[label = \small {$y$}, labelstyle={auto=right, fill=none}](v2)(v3)
\Edge[label = \small {$y$}, labelstyle={auto=right, fill=none}](v4)(v6)
\Edge[label = \small {$z$}, labelstyle={auto=right, fill=none}](v6)(v2)
\Edge[label = \small {$x$}, labelstyle={auto=right, fill=none}](v6)(v8)
\Edge[label = \small {$x$}, labelstyle={auto=right, fill=none}](v3)(v4)
\Edge[label = \small {$x$}, labelstyle={auto=right, fill=none}](v0)(v5)
\Edge[label = \small {$z$}, labelstyle={auto=right, fill=none}](v0)(v4)
\Edge[label = \small {$z$}, labelstyle={auto=right, fill=none}](v1)(v7)
\Edge[label = \small {$x$}, labelstyle={auto=right, fill=none}](v9)(v13)
\Edge[label = \small {$z$}, labelstyle={auto=right, fill=none}](v9)(v14)
\Edge[label = \small {$y$}, labelstyle={auto=right, fill=none}](v10)(v9)
\Edge[label = \small {$z$}, labelstyle={auto=right, fill=none}](v10)(v16)
\Edge[label = \small {$x$}, labelstyle={auto=right, fill=none}](v11)(v10)
\Edge[label = \small {$y$}, labelstyle={auto=right, fill=none}](v11)(v12)
\Edge[label = \small {$z$}, labelstyle={auto=right, fill=none}](v12)(v13)
\Edge[label = \small {$y$}, labelstyle={auto=right, fill=none}](v13)(v15)
\Edge[label = \small {$z$}, labelstyle={auto=right, fill=none}](v15)(v11)
\Edge[label = \small {$x$}, labelstyle={auto=right, fill=none}](v15)(v17)
\end{tikzpicture}
\caption{How to extend a coloring of $G\setminus H_1$ to $G$ when $H_1$ has
order 4 and a 3-vertex $v$ has two neighbors in $H_1$.}
\label{order4-pic}
\end{center}
\end{figure}

Claim~\ref{clm4} completes the proof of Lemma~\ref{lem:summary}, which ends this section.

\section{Reducibility}
\label{fixer-reducibility}


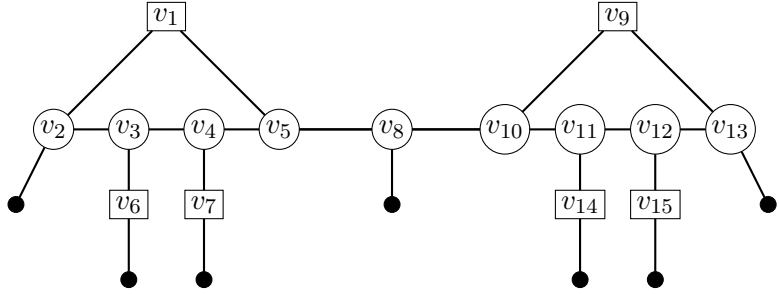
\begin{figure}[bht]
\begin{center}
\begin{tikzpicture}[scale = 10]
\tikzstyle{VertexStyle} = []
\tikzstyle{EdgeStyle} = []
\tikzstyle{labeledStyle}=[shape = circle, minimum size = 6pt, inner sep = 1.2pt, draw]
\tikzstyle{unlabeledStyle}=[shape = circle, minimum size = 6pt, inner sep = 1.2pt, draw, fill]
\Vertex[style = labeledThree, x = 0.80, y = 0.70, L = \small {$v_{11}$}]{v0}
\Vertex[style = labeledThree, x = 0.70, y = 0.70, L = \small {$v_{10}$}]{v1}
\Vertex[style = labeledThree, x = 1.00, y = 0.70, L = \small {$v_{13}$}]{v2}
\Vertex[style = labeledTwo, x = 0.85, y = 0.85, L = \small {$v_9$}]{v3}
\Vertex[style = labeledThree, x = 0.90, y = 0.70, L = \small {$v_{12}$}]{v4}
\Vertex[style = labeledTwo, x = 0.80, y = 0.60, L = \small {$v_{14}$}]{v5}
\Vertex[style = labeledTwo, x = 0.90, y = 0.60, L = \small {$v_{15}$}]{v6}
\Vertex[style = labeledThree, x = 0.20, y = 0.70, L = \small {$v_{3}$}]{v7}
\Vertex[style = labeledThree, x = 0.10, y = 0.70, L = \small {$v_{2}$}]{v8}
\Vertex[style = labeledThree, x = 0.40, y = 0.70, L = \small {$v_{5}$}]{v9}
\Vertex[style = labeledTwo, x = 0.25, y = 0.85, L = \small {$v_1$}]{v10}
\Vertex[style = labeledThree, x = 0.30, y = 0.70, L = \small {$v_{4}$}]{v11}
\Vertex[style = labeledTwo, x = 0.20, y = 0.60, L = \small {$v_{6}$}]{v12}
\Vertex[style = labeledTwo, x = 0.30, y = 0.60, L = \small {$v_{7}$}]{v13}
\Vertex[style = labeledThree, x = 0.55, y = 0.70, L = \small {$v_{8}$}]{v14}
\Vertex[style = unlabeledStyle, x = 0.55, y = 0.60, L = \small {}]{v15}
\Vertex[style = unlabeledStyle, x = 1.05, y = 0.60, L = \small {}]{v16}
\Vertex[style = unlabeledStyle, x = 0.05, y = 0.60, L = \small {}]{v17}
\Vertex[style = unlabeledStyle, x = 0.20, y = 0.50, L = \small {}]{v18}
\Vertex[style = unlabeledStyle, x = 0.30, y = 0.50, L = \small {}]{v19}
\Vertex[style = unlabeledStyle, x = 0.80, y = 0.50, L = \small {}]{v20}
\Vertex[style = unlabeledStyle, x = 0.90, y = 0.50, L = \small {}]{v21}
\Edge[label = \small {}, labelstyle={auto=right, fill=none}](v1)(v0)
\Edge[label = \small {}, labelstyle={auto=right, fill=none}](v3)(v1)
\Edge[label = \small {}, labelstyle={auto=right, fill=none}](v3)(v2)
\Edge[label = \small {}, labelstyle={auto=right, fill=none}](v4)(v2)
\Edge[label = \small {}, labelstyle={auto=right, fill=none}](v0)(v4)
\Edge[label = \small {}, labelstyle={auto=right, fill=none}](v5)(v0)
\Edge[label = \small {}, labelstyle={auto=right, fill=none}](v6)(v4)
\Edge[label = \small {}, labelstyle={auto=right, fill=none}](v7)(v11)
\Edge[label = \small {}, labelstyle={auto=right, fill=none}](v8)(v7)
\Edge[label = \small {}, labelstyle={auto=right, fill=none}](v10)(v8)
\Edge[label = \small {}, labelstyle={auto=right, fill=none}](v10)(v9)
\Edge[label = \small {}, labelstyle={auto=right, fill=none}](v11)(v9)
\Edge[label = \small {}, labelstyle={auto=right, fill=none}](v12)(v7)
\Edge[label = \small {}, labelstyle={auto=right, fill=none}](v13)(v11)
\Edge[label = \small {}, labelstyle={auto=right, fill=none}](v14)(v1)
\Edge[label = \small {}, labelstyle={auto=right, fill=none}](v14)(v9)
\Edge[label = \small {}, labelstyle={auto=right, fill=none}](v15)(v14)
\Edge[label = \small {}, labelstyle={auto=right, fill=none}](v16)(v2)
\Edge[label = \small {}, labelstyle={auto=right, fill=none}](v17)(v8)
\Edge[label = \small {}, labelstyle={auto=right, fill=none}](v20)(v5)
\Edge[label = \small {}, labelstyle={auto=right, fill=none}](v21)(v6)
\Edge[label = \small {}, labelstyle={auto=right, fill=none}](v19)(v13)
\Edge[label = \small {}, labelstyle={auto=right, fill=none}](v18)(v12)
\Edge[label = \small {}, labelstyle={auto=right, fill=none}](v14)(v9)
\Edge[label = \small {}, labelstyle={auto=right, fill=none}](v1)(v14)
\end{tikzpicture}
\end{center}
\caption{A subgraph forbidden from appearing in 3-critical graph $G$.
\label{two-jellyfish-pic2}}
\end{figure}

\begin{lem}
The subgraph shown in Figure~\ref{two-jellyfish-pic2} 
cannot appear in a triangle-free 3-critical graph.  Nor can it appear if we
identify one or two vertex pairs in $\{v_6,v_7,v_{14},v_{15}\}$.  Thus, no rich
vertex has neighbors in two distinct poor fragments each of order at least 4.
\label{two-jellyfish}
\end{lem}

\begin{proof}
We first consider the case where no pairs of 2-vertices are identified.
Let $L$ and $R$ denote the subgraphs of Figure~5 induced by vertices
$v_1,\ldots,v_7$ and $v_9, \ldots,v_{15}$, respectively.
Note that the $L$ and $R$ are symmetric.
By criticality, construct
a partial 3-coloring $G-E(L)$ 
The four vertices where colored and uncolored edges meet are $v_2$, $v_6$,
$v_7$, $v_5$, each of which sees exactly one color.
We have numerous possibilities for the ordered 4-tuple of colors seen by these
vertices (up to permuting color classes, we have 14 such possibilities).
To show that we can extend these partial colorings to $G$, we can assume by
permuting color classes that $v_2$ sees $x$.  We begin by showing that we can
extend the coloring to all of $G$ unless the ordered 4-tuple of colors seen by
$v_2, v_6, v_7, v_5$ is $(x,x,y,y)$ or $(x,y,y,x)$. 

If $v_2$ and $v_5$ see distinct colors, and the ordered 4-tuple of colors seen
by $v_2$, $v_6$, $v_7$, $v_5$ is not of the form $(x,x,y,y)$, then we can use
Figure~\ref{two-jellyfish-colorings-pic}(a) (possibly with colors permuted). If
$v_7$ and $v_5$ see distinct colors, then we extend the coloring using
Figure~\ref{two-jellyfish-colorings-pic}(a) and afterward color
greedily in order along the path $v_6v_3v_2v_1$ of uncolored edges. If not,
then we can assume that $v_2$ and $v_6$ see distinct colors, so we can swap
the roles of $v_2$ and $v_5$.  More formally, we reflect the coloring in
Figure~\ref{two-jellyfish-colorings-pic}(a) across a vertical line running
through $v_1$.

Now suppose instead that $v_2$ and $v_5$ see the same color, $x$.  If $v_6$ or
$v_7$ sees $x$, then we can extend the coloring as in
Figure~\ref{two-jellyfish-colorings-pic}(b): now color greedily along the
path of uncolored edges, ending at $v_7$ (again, if $v_6$ see $x$, but $v_7$
does not, reflect the coloring across a vertical line through $v_1$).  
Further, if $v_6$ and $v_7$ see distinct colors (other than $x$), then we can
color as in Figure~\ref{two-jellyfish-colorings-pic}(c).  Thus, we conclude
that we can extend the partial coloring to $G$ unless the ordered 4-tuple of
colors seen is $(x,x,y,y)$ or $(x,y,y,x)$.  Note that these two bad
possibilities differ in the colors used on \emph{two} pendant edges, even up to
all permutations of color classes.  Thus, if at least one pendant edge is not
yet colored, we can always find an extension of the partial coloring.

Now suppose that $G$ contains a copy of Figure~\ref{two-jellyfish-pic2}.
By criticality, we get a 3-coloring of all of $G$ except the edges with both
endpoints in $v_1,\ldots,v_{15}$. 
Our goal is to color the two remaining edges incident to $v_8$ so that both $L$
and $R$ can be colored as above.  As we already noted, we must thus color $v_5v_8$ and
$v_8v_{10}$ so that neither $L$ or $R$ has as its ordered 4-tuple of colors seen
either $(x,x,y,y)$ or $(x,y,y,x)$.  

\begin{figure}
\begin{center}
\subfloat[]{
\makebox[.3\textwidth]{
\begin{tikzpicture}[scale = 10]
\tikzstyle{VertexStyle} = []
\tikzstyle{EdgeStyle} = []
\tikzstyle{unlabeledStyle}=[shape = circle, minimum size = 6pt, inner sep = 1.2pt, draw, fill]
\Vertex[style = labeledThree, x = 0.55, y = 0.70, L = \small {$v_3$}]{v0}
\Vertex[style = labeledThree, x = 0.45, y = 0.70, L = \small {$v_2$}]{v1}
\Vertex[style = labeledThree, x = 0.75, y = 0.70, L = \small {$v_5$}]{v2}
\Vertex[style = labeledTwo, x = 0.60, y = 0.85, L = \small {$v_1$}]{v3}
\Vertex[style = labeledThree, x = 0.65, y = 0.70, L = \small {$v_4$}]{v4}
\Vertex[style = labeledTwo, x = 0.55, y = 0.60, L = \small {$v_6$}]{v5}
\Vertex[style = labeledTwo, x = 0.65, y = 0.60, L = \small {$v_7$}]{v6}
\Vertex[style = unlabeledStyle, x = 0.40, y = 0.60, L = \small {}]{v7}
\Vertex[style = unlabeledStyle, x = 0.55, y = 0.50, L = \small {}]{v8}
\Vertex[style = unlabeledStyle, x = 0.65, y = 0.50, L = \small {}]{v9}
\Vertex[style = unlabeledStyle, x = 0.80, y = 0.60, L = \small {}]{v10}
\Edge[label = \small {}, labelstyle={auto=right, fill=none}](v1)(v0)
\Edge[label = \small {$x$}, labelstyle={auto=right, fill=none}](v2)(v3)
\Edge[label = \small {$z$}, labelstyle={auto=right, fill=none}](v2)(v4)
\Edge[label = \small {}, labelstyle={auto=right, fill=none}](v3)(v1)
\Edge[label = \small {$x$}, labelstyle={auto=right, fill=none}](v4)(v0)
\Edge[label = \small {}, labelstyle={auto=right, fill=none}](v5)(v0)
\Edge[label = \small {$y$}, labelstyle={auto=right, fill=none}](v6)(v4)
\Edge[label = \small {$x$}, labelstyle={auto=right, fill=none}](v1)(v7)
\Edge[label = \small {$*$}, labelstyle={auto=right, fill=none}](v8)(v5)
\Edge[label = \small {$\ne y$}, labelstyle={auto=right, fill=none}](v9)(v6)
\Edge[label = \small {$y$}, labelstyle={auto=right, fill=none}](v10)(v2)
\end{tikzpicture}
}}
~~
\subfloat[]{
\makebox[.3\textwidth]{
\begin{tikzpicture}[scale = 10]
\tikzstyle{VertexStyle} = []
\tikzstyle{EdgeStyle} = []
\tikzstyle{unlabeledStyle}=[shape = circle, minimum size = 6pt, inner sep = 1.2pt, draw, fill]
\Vertex[style = labeledThree, x = 0.55, y = 0.70, L = \small {$v_3$}]{v0}
\Vertex[style = labeledThree, x = 0.45, y = 0.70, L = \small {$v_2$}]{v1}
\Vertex[style = labeledThree, x = 0.75, y = 0.70, L = \small {$v_5$}]{v2}
\Vertex[style = labeledTwo, x = 0.60, y = 0.85, L = \small {$v_1$}]{v3}
\Vertex[style = labeledThree, x = 0.65, y = 0.70, L = \small {$v_4$}]{v4}
\Vertex[style = labeledTwo, x = 0.55, y = 0.60, L = \small {$v_6$}]{v5}
\Vertex[style = labeledTwo, x = 0.65, y = 0.60, L = \small {$v_7$}]{v6}
\Vertex[style = unlabeledStyle, x = 0.40, y = 0.60, L = \small {}]{v7}
\Vertex[style = unlabeledStyle, x = 0.55, y = 0.50, L = \small {}]{v8}
\Vertex[style = unlabeledStyle, x = 0.65, y = 0.50, L = \small {}]{v9}
\Vertex[style = unlabeledStyle, x = 0.80, y = 0.60, L = \small {}]{v10}
\Edge[label = \small {}, labelstyle={auto=right, fill=none}](v1)(v0)
\Edge[label = \small {}, labelstyle={auto=right, fill=none}](v3)(v1)
\Edge[label = \small {}, labelstyle={auto=right, fill=none}](v3)(v2)
\Edge[label = \small {}, labelstyle={auto=right, fill=none}](v4)(v2)
\Edge[label = \small {}, labelstyle={auto=right, fill=none}](v5)(v0)
\Edge[label = \small {}, labelstyle={auto=right, fill=none}](v6)(v4)
\Edge[label = \small {$x$}, labelstyle={auto=right, fill=none}](v1)(v7)
\Edge[label = \small {$*$}, labelstyle={auto=right, fill=none}](v8)(v5)
\Edge[label = \small {$x$}, labelstyle={auto=right, fill=none}](v9)(v6)
\Edge[label = \small {$x$}, labelstyle={auto=right, fill=none}](v10)(v2)
\Edge[label = \small {$x$}, labelstyle={auto=right, fill=none}](v4)(v0)
\end{tikzpicture}
}}
~~
\subfloat[]{
\makebox[.3\textwidth]{
\begin{tikzpicture}[scale = 10]
\tikzstyle{VertexStyle} = []
\tikzstyle{EdgeStyle} = []
\tikzstyle{unlabeledStyle}=[shape = circle, minimum size = 6pt, inner sep = 1.2pt, draw, fill]
\Vertex[style = labeledThree, x = 0.55, y = 0.70, L = \small {$v_3$}]{v0}
\Vertex[style = labeledThree, x = 0.45, y = 0.70, L = \small {$v_2$}]{v1}
\Vertex[style = labeledThree, x = 0.75, y = 0.70, L = \small {$v_5$}]{v2}
\Vertex[style = labeledTwo, x = 0.60, y = 0.85, L = \small {$v_1$}]{v3}
\Vertex[style = labeledThree, x = 0.65, y = 0.70, L = \small {$v_4$}]{v4}
\Vertex[style = labeledTwo, x = 0.55, y = 0.60, L = \small {$v_6$}]{v5}
\Vertex[style = labeledTwo, x = 0.65, y = 0.60, L = \small {$v_7$}]{v6}
\Vertex[style = unlabeledStyle, x = 0.40, y = 0.60, L = \small {}]{v7}
\Vertex[style = unlabeledStyle, x = 0.55, y = 0.50, L = \small {}]{v8}
\Vertex[style = unlabeledStyle, x = 0.65, y = 0.50, L = \small {}]{v9}
\Vertex[style = unlabeledStyle, x = 0.80, y = 0.60, L = \small {}]{v10}
\Edge[label = \small {$y$}, labelstyle={auto=right, fill=none}](v1)(v0)
\Edge[label = \small {$z$}, labelstyle={auto=right, fill=none}](v3)(v1)
\Edge[label = \small {$z$}, labelstyle={auto=right, fill=none}](v5)(v0)
\Edge[label = \small {$y$}, labelstyle={auto=right, fill=none}](v6)(v4)
\Edge[label = \small {$x$}, labelstyle={auto=right, fill=none}](v1)(v7)
\Edge[label = \small {$y$}, labelstyle={auto=right, fill=none}](v8)(v5)
\Edge[label = \small {$z$}, labelstyle={auto=right, fill=none}](v9)(v6)
\Edge[label = \small {$x$}, labelstyle={auto=right, fill=none}](v10)(v2)
\Edge[label = \small {$x$}, labelstyle={auto=right, fill=none}](v4)(v0)
\Edge[label = \small {$y$}, labelstyle={auto=right, fill=none}](v2)(v3)
\Edge[label = \small {$z$}, labelstyle={auto=right, fill=none}](v2)(v4)
\end{tikzpicture}
}}
\end{center}
\caption{Extensions for part of Figure~\ref{two-jellyfish-pic2}, based on the
colors seen by $v_2$, $v_5$, $v_6$, and $v_7$. \label{two-jellyfish-colorings-pic}}
\end{figure}
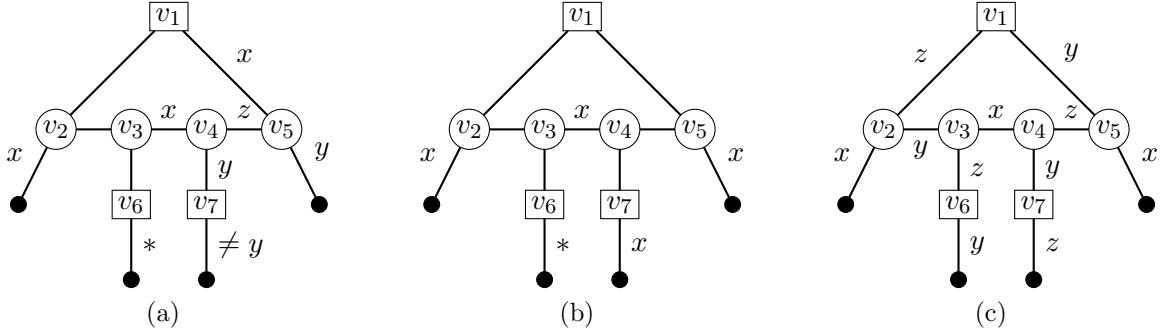

Given the colors seen by $v_2$, $v_6$, $v_7$, at most one choice of color for $v_5v_8$ gives a bad 4-tuple
for $L$.  Similarly, at most one choice of color for $v_8v_{10}$ gives a bad
4-tuple for $R$.  We can color the edges as desired unless the color
that is bad on $v_5v_8$ for $L$ is the same as the color that is bad
on $v_8v_{10}$ for $R$, and that color, say $x$, is different from 
the color $y$ seen by $v_8$.  So suppose this is true.  
Now perform an $(x,y)$-swap at $v_8$.  If this Kempe chain ends in neither
$L$ nor $R$, then we color $v_5v_8$ and $v_8v_{10}$ arbitrarily.
Now we can extend the coloring to both $L$ and $R$.  So suppose instead that
the Kempe chain ends in $L$ (by symmetry).  Now we can color $L$, since $v_5v_8$ is
uncolored.  Afterward, the color for $v_8v_{10}$ is determined, and we can color
$R$.  This completes the case where no pairs of 2-vertices are identified.
\bigskip

Now we consider the case where two vertex pairs in $\{v_6,v_7,v_{14},v_{15}\}$
are identified.  Since $G$ has no 3-cycles, each of $v_6,v_7$ must be identified
with one of $v_{14},v_{15}$.  By criticality, color all edges except those with
both endpoints in $v_1,\ldots,v_{15}$.  
Note that we have only 3 incident colored edges. 
As above, we now extend the coloring to $G$, using
Figure~\ref{two-jellyfish-colorings-pic}.

Suppose that all colored incident edges use the same color, say $x$.
Now color $v_5v_8$ and $v_4v_7$ with $y$ and color $v_8v_{10}$
and $v_{11}v_{14}$ with $z$.  (Perhaps $v_7=v_{14}$, but this is okay.)  Now we
can extend the coloring to each side,
as in Figure~\ref{two-jellyfish-colorings-pic}(a).  A similar strategy works in
every case except when $v_2$ and $v_{13}$ see a common color, say $x$, and $v_8$
sees some other color, say $y$.
We always color $v_5v_8$ and $v_8v_{10}$ so that
their colors differ from those seen by $v_2$ and $v_{13}$, respectively.
Next, we color $v_4v_7$ and $v_{11}v_{14}$ to match $v_5v_8$ and $v_8v_{10}$,
respectively.  Finally, we can color each side as in
Figure~\ref{two-jellyfish-colorings-pic}(a), reflected.  
%
%
So suppose that $v_2$ and $v_{13}$ see $x$ and $v_8$ sees $y$.  Color $v_5v_8$
with $x$ and $v_8v_{10}$ with $z$.  Now color $v_3v_4$ with $x$ and $v_{14}v_i$
with $y$, where $i\in\{3,4\}$.  To extend $L$, we assume by symmetry that
$v_{14}=v_7$, so $v_7v_4$ is colored $y$ and $v_3v_6$ is uncolored.  Greedily
color the path $v_4v_5v_1v_2v_3v_6$, starting from $v_4$.  Now since $v_{10}$
and $v_{14}$ see different colors, we extend $R$ as in
Figure~\ref{two-jellyfish-colorings-pic}(a).  
This completes the case of two pairs of identified vertices.
\bigskip

Now suppose that one vertex pair in $\{v_6,v_7,v_{14},v_{15}\}$ is
identified; we consider three cases.  The identified pair is either
$(v_6,v_{15})$, $(v_7,v_{14})$, or $(v_6,v_{14})$; we call these cases
``outside'', ``inside'', and ``mixed''.  In each case, five vertices in
$v_1,\ldots, v_{15}$ see colors, but we initially consider only the colors seen
by $v_2$, $v_8$, and $v_{13}$.  Thus, for example, we write the 3-tuple
$(x,y,z)$ to signify that $v_2$ sees $x$, $v_8$ sees $y$, and $v_{13}$ sees $z$.

First consider outside.  
Suppose we have the 3-tuple $(x,x,y)$.  Color $v_3v_6$ with $x$, color $v_5v_8$ and
$v_{12}v_{15}$ with $y$, and color $v_8v_{10}$ with $z$.
We can extend the
coloring on each side as in Figure~\ref{two-jellyfish-colorings-pic}(a).
A similar strategy works for 3-tuples $(x,y,y)$ and $(x,y,z)$.  Consider instead
$(x,y,x)$.  Now color $v_5v_8$ and $v_{12}v_{15}$ with $x$ and color
$v_8v_{10}$ with $z$.  We can color $R$ as in
Figure~\ref{two-jellyfish-colorings-pic}(a) and $L$ as in 
Figure~\ref{two-jellyfish-colorings-pic}(b).  Finally, consider the 3-tuple $(x,x,x)$.
If $v_7$ sees $x$, then color $v_{12}v_{15}$ with $x$,
color $v_5v_8$ with $y$, and color $v_8v_{10}$ with $z$.  Now
color both $L$ and $R$ as in Figure~\ref{two-jellyfish-colorings-pic}(a).  
Otherwise, by symmetry $v_7$ sees $y$.  
Now color $v_{12}v_{15}$ with $x$,  $v_8v_{10}$ with $y$, and $v_5v_8$ with $z$. 
We can again color both $L$ and $R$ as in Figure~\ref{two-jellyfish-colorings-pic}(a).  

Now consider inside.
This case is similar to above.  
Suppose that $v_2$, $v_8$, $v_{13}$ see some ordered triple \emph{other} than $(x,y,x)$.
Color $v_5v_8$ and $v_8v_{10}$ to differ from the colors seen by $v_2$ and
$v_{13}$, respectively.  Now color $v_4v_7$ and $v_{11}v_{14}$ to match
$v_5v_8$ and $v_8v_{10}$, respectively.  Finally, color both $L$ and $R$ as in
Figure~\ref{two-jellyfish-colorings-pic}(a).  So suppose $v_2$, $v_8$, $v_{13}$
sees $(x,y,x)$.
If $v_6$ sees $x$, then color
$v_5v_8$ with $x$ and $v_8v_{10}$ with $z$.  Now we can color $R$ first,
then color $L$, as in Figure~\ref{two-jellyfish-colorings-pic}(b), reflected. 
So assume $v_6$ does not see $x$.  Now color $v_5v_8$ with $z$ and $v_8v_{10}$ with
$x$. Color $R$ first, then color $L$ as in
Figure~\ref{two-jellyfish-colorings-pic}(a) reflected, since $v_2$ and $v_6$ see distinct
colors.

Finally, consider mixed.  Recall that $v_6$ and $v_{14}$ are identified.
Suppose the triple seen by $v_2$, $v_8$, $v_{13}$ is something other than
other than $(x,x,x)$ and $(x,y,y)$.  If $v_{13}$ and
$v_{15}$ see distinct colors, then color $v_8v_{10}$ with a color not seen by
$v_{13}$.  Now color $L$, then extend to $R$ as in
Figure~\ref{two-jellyfish-colorings-pic}(a).  Otherwise $v_{13}$ and $v_{15}$
see the same color, so use that color on $v_8v_{10}$.  Now color $L$, then
extend to $R$,  as in Figure~\ref{two-jellyfish-colorings-pic}(b).
Instead, consider $(x,x,x)$.  Color $v_3v_6$ with $x$,
color $v_5v_8$ with $y$, and both $v_8v_{10}$ and $v_{11}v_{14}$ with $z$.  Now
extend both sides as in Figure~\ref{two-jellyfish-colorings-pic}(a).  
Finally, consider $(x,y,y)$.  If $v_{15}$ sees a color other than $y$, then
color $v_8v_{10}$ to avoid the color seen by $v_{15}$.  Now color $L$,
followed by $R$, as in Figure~\ref{two-jellyfish-colorings-pic}(a).
Similarly, if $v_7$ sees $x$, then color $v_5v_8$ with
$x$ and color $R$, followed by $L$, asi in
Figure~\ref{two-jellyfish-colorings-pic}(a)  Likewise, if $v_7$ sees $y$,
then color $v_5v_8$ with $z$ and color $R$, followed by $L$.
Thus, we conclude that $v_7$ sees $z$ and $v_{15}$ sees $y$.  Now perform an
$(x,y)$-swap at $v_8$.  The resulting coloring will be one of the cases above.
\end{proof}


\begin{lem}
Let $G'$ be a 3-colorable graph having the configuration in
Figure~\ref{fig:KempePath} as a subgraph.  Now $G'$ has a 3-coloring such that
at least one of the Kempe chains through edge $w'w$ is a path, not a cycle.
\label{lem:new1}
\end{lem}
\begin{proof}
By possibly permuting colors, assume that the three edges incident with $w$ are
given the colors indicated; also assume that, for every 3-coloring of $G'$ with
these colors, both the $(x,y)$-chain and the $(x,z)$-chain through edges $w'w$
are cycles.  This implies that swapping colors in any $(x,y)$-chain or
$(x,z)$-chain that is not a cycle will not change the color of any edge at $w$.
Let the unshown neighbors of $u_1$, $u_2$, $u_3$, $v_3$ be $u'_1$, $u'_2$,
$u'_3$, $v'_3$, respectively.

Let $C_{x,y}$ denote the $(x,y)$-Kempe chain (which must be a cycle) containing
the path $w'wv_1$.  If $C_{x,y}$ includes $v_1u_1u'_1$, then an $(x,z)$-swap at
$u_1$ will reroute it along $v_1v_2$; so we assume that $C_{x,y}$ includes
$v_1v_2$.  So $v_1u_1$ has color $z$.

We now prove the following: \textit{If $C_{x,y}$ contains the path
$w'wv_1v_2$, then (a) $C_{x,y}$ contains edge $v_2v_3$, and (b) edges $u_1u'_1$
and $u_2u'_2$ have the same color and are the end-edges of an $(x,y)$-chain.}

To prove (a), suppose to the contrary that $C_{x,y}$ includes $v_2u_2u'_2$.  By
an $(x,y)$-swap at $u_1$ if necessary, we can assume that $u_1u'_1$ has color
$x$.  If the $(x,z)$-chain $P_{x,z}(u_2)$ starting along $u_2u'_2$ does not
include the path $u'_1u_1v_1v_2v_3$, then swapping colors on $P_{x,z}(u_2)$ will
terminate $C_{x,y}$ at $u_2$.  And if $P_{x,z}(u_2)$ does include this path,
then swapping colors on $P_{x,z}(u_2)$ will reroute $C_{x,y}$ along $w'wv_1u_1$,
where it terminates.  Both of these possibilities contradict our assumptions,
and this proves (a).  Thus edges $v_1u_1$ and $v_2u_2$ are both colored $z$.

Suppose now that (b) is false.  Now by $(x,y)$-interchanges at $u_1$ and/or
$u_2$, if necessary, we can make $u_1u'_1$ and $u_2u'_2$ have colors $x$ and
$y$, respectively.  Now an $(x,z)$-swap at $u_2$ reroutes $C_{x,y}$ along
$v_1u_1$, where it terminates.  This contradiction proves (b).  By swapping
colors in the $(x,y)$-chain linking $u_1u'_1$ and $u_2u'_2$, if necessary,
assume that these edges both have color $y$.  Note that an $(x,z)$-interchange
at $u_3$ now cannot change the color of edge $v_1v_2$.

If $C_{x,y}$ includes $v_3u_3u'_3$, then an $(x,z)$-swap at $u_3$ will reroute
it along $v_3v'_3$; so we will assume that $C_{x,y}$ contains the path
$w'wv_1v_2v_3v'_3$ and that all three edges $v_iu_i$ are colored $z$.  By (b)
applied the the current coloring, there is still an $(x,y)$-chain linking
$u_1u'_1$ and $u_2u'_2$.  Thus the $(x,y)$-chains containing $u_2u'_2$ and
$u_3u'_3$ are disjoint.  By swapping colors on these chains if necessary, we can
assume that both of these edges have color $x$.  Now interchanging colors $y$
and $z$ on the Kempe chain $u_2v_2v_3u_3$ will reroute $C_{x,y}$ along $v_2u_2$,
which contradicts (a).  This contradiction completes the proof.
\end{proof}

\begin{figure}
\centering
\begin{tikzpicture}[scale = 15]
\tikzstyle{VertexStyle} = []
\tikzstyle{EdgeStyle} = []
\tikzstyle{labeledStyle}=[shape = circle, minimum size = 6pt, inner sep = 1.2pt, draw]
\tikzstyle{unlabeledStyle}=[shape = circle, minimum size = 6pt, inner sep = 1.2pt, draw, fill]
\Vertex[style = labeledThree, x = 0.550, y = 0.700, L = \small {$w'$}]{v0}
\Vertex[style = labeledThree, x = 0.650, y = 0.700, L = \small {$\,w\,$}]{v1}
\Vertex[style = labeledThree, x = 0.750, y = 0.700, L = \small {$v_1$}]{v2}
\Vertex[style = labeledThree, x = 0.850, y = 0.700, L = \small {$v_2$}]{v3}
\Vertex[style = labeledThree, x = 0.950, y = 0.700, L = \small {$v_3$}]{v4}
\Vertex[style = unlabeledStyle, x = 1.050, y = 0.700, L = \small {}]{v5}
\Vertex[style = labeledTwo, x = 0.950, y = 0.800, L = \small {$u_3$}]{v6}
\Vertex[style = unlabeledStyle, x = 0.950, y = 0.900, L = \small {}]{v7}
\Vertex[style = unlabeledStyle, x = 0.850, y = 0.900, L = \small {}]{v8}
\Vertex[style = labeledTwo, x = 0.850, y = 0.800, L = \small {$u_2$}]{v9}
\Vertex[style = labeledTwo, x = 0.750, y = 0.800, L = \small {$u_1$}]{v10}
\Vertex[style = unlabeledStyle, x = 0.750, y = 0.900, L = \tiny {}]{v11}
\Vertex[style = unlabeledStyle, x = 0.650, y = 0.800, L = \tiny {}]{v12}
\Vertex[style = unlabeledStyle, x = 0.450, y = 0.700, L = \tiny {}]{v13}
\Vertex[style = unlabeledStyle, x = 0.550, y = 0.800, L = \tiny {}]{v14}
\Edge[label = \small {$x$}, labelstyle={auto=left, fill=none}](v1)(v0)
\Edge[label = \small {$y$}, labelstyle={auto=right, fill=none}](v1)(v2)
\Edge[label = \tiny {}, labelstyle={auto=right, fill=none}](v3)(v4)
\Edge[label = \tiny {}, labelstyle={auto=right, fill=none}](v3)(v2)
\Edge[label = \tiny {}, labelstyle={auto=right, fill=none}](v5)(v4)
\Edge[label = \tiny {}, labelstyle={auto=right, fill=none}](v6)(v7)
\Edge[label = \tiny {}, labelstyle={auto=right, fill=none}](v6)(v4)
\Edge[label = \tiny {}, labelstyle={auto=right, fill=none}](v9)(v8)
\Edge[label = \tiny {}, labelstyle={auto=right, fill=none}](v9)(v3)
\Edge[label = \tiny {}, labelstyle={auto=right, fill=none}](v10)(v11)
\Edge[label = \tiny {}, labelstyle={auto=right, fill=none}](v10)(v2)
\Edge[label = \small {$z$}, labelstyle={auto=right, fill=none}](v12)(v1)
\Edge[label = \tiny {}, labelstyle={auto=right, fill=none}](v0)(v14)
\Edge[label = \tiny {}, labelstyle={auto=right, fill=none}](v0)(v13)
\end{tikzpicture}
\caption{If a graph contains this configuration and has a 3-coloring, then it
has a 3-coloring where one Kempe chain through edge $w'w$ is a path, not a cycle.
\label{fig:KempePath}
}
\end{figure}
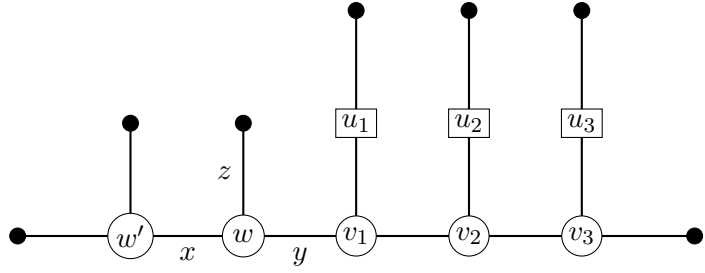

\begin{lem}
The configuration shown in Figure~\ref{fig:2a} is not a
subgraph of any 3-critical graph.
\label{lem:new2}
\label{lem:fig2a}
\end{lem}
\begin{proof}
Suppose $G$ is a 3-cirtical graph containing this configuration as a subgraph.
Let $w'$ be the unshown neighbor of $u$ in Figure~\ref{fig:2a}.
Since $G$ is 3-critical, $G-uw$ has a 3-coloring
that cannot be extended to a 3-coloring of $G$, and which therefore gives a
3-coloring of the graph $G'$ obtained from $G$ by contracting edge $uw$ (so
replacing path $w'uw$ by a single edge $w'w$).

By Lemma~\ref{lem:new1}, $G'$ has a 3-coloring such that one of the Kempe chains
through $w'w$ is a path $P$.  Transferring this coloring to $G$, it is 
proper at every vertex except $u$, where the edges $w'u$ and $uw$ have the same
color.  However, swapping colors in a maximal segment of $P$ ending at $u$
creates a proper coloring of $G$, and this contradiction completes the proof.
\end{proof}

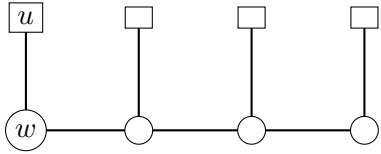
\begin{figure}
\centering
\begin{tikzpicture}[scale = 15]
\tikzstyle{VertexStyle} = []
\tikzstyle{EdgeStyle} = []
\tikzstyle{labeledThree}=[shape = circle, minimum size = 8pt, inner sep = 2.2pt, draw]
\tikzstyle{labeledTwo}=[shape = rectangle, minimum size = 8pt, inner sep = 3.2pt, draw]
\tikzstyle{unlabeledStyle}=[shape = circle, minimum size = 6pt, inner sep = 1.2pt, draw, fill]
\Vertex[style = labeledTwo, x = 0.45, y = 0.75, L = \small {$u$}]{v0}
\Vertex[style = labeledThree, x = 0.45, y = 0.65, L = \small {$w$}]{v1}
\Vertex[style = labeledThree, x = 0.55, y = 0.65, L = \Large {$~$}]{v2}
\Vertex[style = labeledTwo, x = 0.55, y = 0.75, L = \Large {$\>$}]{v3}
\Vertex[style = labeledTwo, x = 0.65, y = 0.75, L = \Large {$\>$}]{v4}
\Vertex[style = labeledThree, x = 0.65, y = 0.65, L = \Large {$~$}]{v5}
\Vertex[style = labeledThree, x = 0.75, y = 0.65, L = \Large {$~$}]{v6}
\Vertex[style = labeledTwo, x = 0.75, y = 0.75, L = \Large {$\>$}]{v7}
\Edge[](v6)(v7)
\Edge[](v6)(v5)
\Edge[](v5)(v4)
\Edge[](v5)(v2)
\Edge[](v2)(v3)
\Edge[](v2)(v1)
\Edge[](v1)(v0)
\end{tikzpicture}
\caption{
This configuration cannot appear in a 3-critical graph.
\label{fig:2a}
}
\end{figure}

\begin{figure}
\centering
\begin{tikzpicture}[scale = 13]
\tikzstyle{VertexStyle} = []
\tikzstyle{EdgeStyle} = []
\tikzstyle{labeledStyle}=[shape = circle, minimum size = 6pt, inner sep = 1.2pt, draw]
\tikzstyle{unlabeledStyle}=[shape = circle, minimum size = 6pt, inner sep = 1.2pt, draw, fill]
\Vertex[style = labeledThree, x = 0.550, y = 0.700, L = \large {$~~$}]{v0}
\Vertex[style = labeledThree, x = 0.450, y = 0.700, L = \large {$~~$}]{v1}
\Vertex[style = labeledThree, x = 0.750, y = 0.700, L = \large {$~~$}]{v2}
\Vertex[style = labeledTwo, x = 0.600, y = 0.850, L = \large {$~~$}]{v3}
\Vertex[style = labeledThree, x = 0.650, y = 0.700, L = \large {$~~$}]{v4}
\Vertex[style = labeledTwo, x = 0.650, y = 0.600, L = \small
{{\mbox{~~}}}]{v5}
\Vertex[style = unlabeledStyle, x = 0.350, y = 0.700, L = \small {}]{v6}
\Vertex[style = unlabeledStyle, x = 0.650, y = 0.500, L = \small {}]{v7}
\Vertex[style = labeledThree, x = 0.850, y = 0.700, L = \large {$~~$}]{v8}
\Vertex[style = labeledTwo, x = 0.700, y = 0.850, L = \large {$~~$}]{v9}
\Vertex[style = unlabeledStyle, x = 0.950, y = 0.700, L = \small {}]{v10}
\Edge[label = \small {$y$}, labelstyle={auto=right, fill=none}](v1)(v0)
\Edge[label = \small {$x$}, labelstyle={auto=left, fill=none}](v1)(v6)
\Edge[label = \small {$x$}, labelstyle={auto=right, fill=none}](v2)(v3)
\Edge[label = \small {$y$}, labelstyle={auto=left, fill=none}](v2)(v4)
\Edge[label = \small {$z$}, labelstyle={auto=right, fill=none}](v3)(v1)
\Edge[label = \small {$x$}, labelstyle={auto=left, fill=none}](v4)(v0)
\Edge[label = \small {$z$}, labelstyle={auto=right, fill=none}](v5)(v4)
\Edge[label = \small {$x$ or $y$}, labelstyle={auto=right, fill=none}](v7)(v5)
\Edge[label = \small {$z$}, labelstyle={auto=left, fill=none}](v8)(v2)
\Edge[label = \small {$z$}, labelstyle={auto=right, fill=none}](v9)(v0)
\Edge[label = \small {$y$ or $x$}, labelstyle={auto=left, fill=none}](v9)(v8)
\Edge[label = \small {$x$ or $y$}, labelstyle={auto=right, fill=none}](v8)(v10)
\end{tikzpicture}

\caption{How to extend a 3-coloring of $G-E(H_1^+)$ to $G$ when $H_1$ is a poor
fragment of order 5 and at least two of its boundary edges have the same color.
\label{fig:order5extension}}
\end{figure}
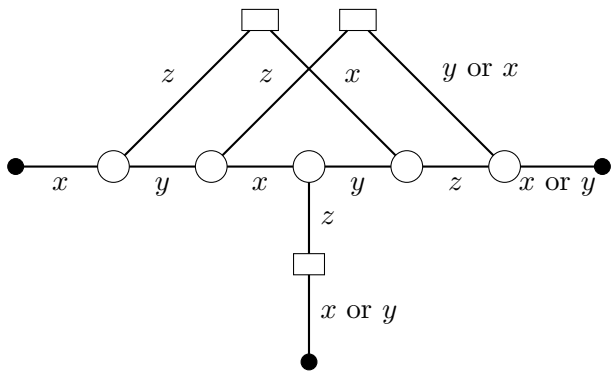

\begin{lem}
Let $G$ be a 3-critical graph. 
No rich vertex of $G$ has
neighbors in three poor fragments with orders 5, 3, and 3.
\label{lem:new3}
\end{lem}
\begin{proof}
Suppose that $w$ is a vertex of $G$ with neighbors in three poor fragments $H_1$,
$H_2$, $H_3$ of orders 5, 3, 3, respectively.  Exactly one 2-vertex has a
single neighbor in $H_1$, and its other neighbor cannot be in both $H_2$
and $H_3$; assume it is not in $H_2$, so that $H_1^+$ and $H_2^+$ are disjoint.
(We will make no further use of $H_3$.)

Consider a 3-coloring of $G-E(H_1^+)$.  There are three colored edges that are
incident with vertices of $H_1^+$; call them \emph{boundary} edges.  It is
easy to see that if any two of them have the same color, then the coloring can
be extended to all of $G$; see Figure~\ref{fig:order5extension}.  This
contradiction shows that the boundary edges must have distinct colors.  Thus,
there is a 3-coloring of the graph $G'$ formed from $G$ by contracting $H_1^+$
to a single vertex $w'$.  Note that $w'$, $w$, and $H_2^+$ form in $G'$ the
configuration shown in Figure~\ref{fig:KempePath}.  By Lemma~\ref{lem:new1},
there is a coloring of $G'$ such that one of the Kempe chains containing $w'w$
is a path $P$.

In $G$, path $P$ consists of two disjoint paths, each ending with one of the
boundary edges.  Swapping colors in one of these paths will create a 3-coloring
of $G-E(H_1^+)$ such that two of the boundary edges have the same color.  As we
have seen, this coloring can be extended to the whole graph $G$, and this
contradiction completes the proof.
\end{proof}

\section{An Improved Bound using a Computer}
\label{computer}

\begin{figure}[ht]
\begin{center}
\subfloat[]{
\makebox[.25\textwidth]{
\begin{tikzpicture}[scale = 6]
\tikzstyle{VertexStyle} = []
\tikzstyle{EdgeStyle} = []
\tikzstyle{labeledStyle}=[shape = circle, minimum size = 6pt, inner sep = 1.2pt, draw]
\tikzstyle{unlabeledStyle}=[shape = circle, minimum size = 6pt, inner sep = 1.2pt, draw, fill]
\Vertex[style = labeledThree, x = 0.600000023841858, y = 0.75, L = \Large {$~$}]{v0}
\Vertex[style = labeledTwo, x = 0.5, y = 0.650000005960464, L = \Large {$\>$}]{v1}
\Vertex[style = labeledThree, x = 0.300000011920929, y = 0.75, L = \Large {$~$}]{v2}
\Vertex[style = labeledThree, x = 0.400000005960464, y = 0.75, L = \Large {$~$}]{v3}
\Vertex[style = labeledTwo, x = 0.300000011920929, y = 0.650000005960464, L = \Large {$\>$}]{v4}
\Vertex[style = labeledThree, x = 0.5, y = 0.849999994039536, L = \Large {$~$}]{v5}
\Edge[label = \Large {}, labelstyle={auto=right, fill=none}](v1)(v0)
\Edge[label = \Large {}, labelstyle={auto=right, fill=none}](v1)(v3)
\Edge[label = \Large {}, labelstyle={auto=right, fill=none}](v3)(v2)
\Edge[label = \Large {}, labelstyle={auto=right, fill=none}](v4)(v2)
\Edge[label = \Large {}, labelstyle={auto=right, fill=none}](v5)(v0)
\Edge[label = \Large {}, labelstyle={auto=right, fill=none}](v5)(v3)
\end{tikzpicture}
}}
\subfloat[]{
\makebox[.25\textwidth]{
\begin{tikzpicture}[scale = 6]
\tikzstyle{VertexStyle} = []
\tikzstyle{EdgeStyle} = []
\tikzstyle{labeledStyle}=[shape = circle, minimum size = 6pt, inner sep = 1.2pt, draw]
\tikzstyle{unlabeledStyle}=[shape = circle, minimum size = 6pt, inner sep = 1.2pt, draw, fill]
\Vertex[style = labeledThree, x = 0.60, y = 0.75, L = \Large {$~$}]{v0}
\Vertex[style = labeledThree, x = 0.70, y = 0.75, L = \Large {$~$}]{v1}
\Vertex[style = labeledThree, x = 0.80, y = 0.75, L = \Large {$~$}]{v2}
\Vertex[style = labeledTwo, x = 0.70, y = 0.65, L = \Large {$\>$}]{v3}
\Vertex[style = labeledTwo, x = 0.60, y = 0.65, L = \Large {$\>$}]{v4}
\Vertex[style = labeledTwo, x = 0.80, y = 0.65, L = \Large {$\>$}]{v5}
\Vertex[style = labeledThree, x = 0.20, y = 0.75, L = \Large {$~$}]{v6}
\Vertex[style = labeledThree, x = 0.30, y = 0.75, L = \Large {$~$}]{v7}
\Vertex[style = labeledThree, x = 0.40, y = 0.75, L = \Large {$~$}]{v8}
\Vertex[style = labeledTwo, x = 0.20, y = 0.65, L = \Large {$\>$}]{v9}
\Vertex[style = labeledThree, x = 0.5, y = 0.85, L = \Large {$~$}]{v10}
\Edge[label = \Large {}, labelstyle={auto=right, fill=none}](v1)(v0)
\Edge[label = \tiny {}, labelstyle={auto=right, fill=none}](v1)(v2)
\Edge[label = \tiny {}, labelstyle={auto=right, fill=none}](v3)(v1)
\Edge[label = \tiny {}, labelstyle={auto=right, fill=none}](v4)(v0)
\Edge[label = \tiny {}, labelstyle={auto=right, fill=none}](v5)(v2)
\Edge[label = \tiny {}, labelstyle={auto=right, fill=none}](v6)(v7)
\Edge[label = \tiny {}, labelstyle={auto=right, fill=none}](v7)(v4)
\Edge[label = \tiny {}, labelstyle={auto=right, fill=none}](v8)(v7)
\Edge[label = \tiny {}, labelstyle={auto=right, fill=none}](v9)(v6)
\Edge[label = \tiny {}, labelstyle={auto=right, fill=none}](v10)(v0)
\Edge[label = \tiny {}, labelstyle={auto=right, fill=none}](v10)(v8)
\end{tikzpicture}
}}
\subfloat[]{
\makebox[.25\textwidth]{
\begin{tikzpicture}[scale = 6]
\tikzstyle{VertexStyle} = []
\tikzstyle{EdgeStyle} = []
\tikzstyle{labeledStyle}=[shape = circle, minimum size = 6pt, inner sep = 1.2pt, draw]
\tikzstyle{unlabeledStyle}=[shape = circle, minimum size = 6pt, inner sep = 1.2pt, draw, fill]
\Vertex[style = labeledThree, x = 0.55, y = 0.75, L = \Large {$~$}]{v0}
\Vertex[style = labeledThree, x = 0.65, y = 0.75, L = \Large {$~$}]{v1}
\Vertex[style = labeledThree, x = 0.75, y = 0.85, L = \Large {$~$}]{v2}
\Vertex[style = labeledThree, x = 0.85, y = 0.75, L = \Large {$~$}]{v3}
\Vertex[style = labeledThree, x = 0.95, y = 0.75, L = \Large {$~$}]{v4}
\Vertex[style = labeledTwo, x = 0.85, y = 0.65, L = \Large {$\>$}]{v5}
\Vertex[style = labeledTwo, x = 0.65, y = 0.65, L = \Large {$\>$}]{v6}
\Vertex[style = labeledThree, x = 0.45, y = 0.75, L = \Large {$~$}]{v7}
\Vertex[style = labeledTwo, x = 0.45, y = 0.65, L = \Large {$\>$}]{v8}
\Edge[label = \Large {}, labelstyle={auto=right, fill=none}](v1)(v0)
\Edge[label = \tiny {}, labelstyle={auto=right, fill=none}](v1)(v2)
\Edge[label = \tiny {}, labelstyle={auto=right, fill=none}](v3)(v2)
\Edge[label = \tiny {}, labelstyle={auto=right, fill=none}](v3)(v4)
\Edge[label = \tiny {}, labelstyle={auto=right, fill=none}](v3)(v5)
\Edge[label = \tiny {}, labelstyle={auto=right, fill=none}](v5)(v0)
\Edge[label = \tiny {}, labelstyle={auto=right, fill=none}](v6)(v1)
\Edge[label = \tiny {}, labelstyle={auto=right, fill=none}](v7)(v0)
\Edge[label = \tiny {}, labelstyle={auto=right, fill=none}](v7)(v8)
\Edge[label = \tiny {}, labelstyle={auto=right, fill=none}](v4)(v8)
\end{tikzpicture}
}}
		\subfloat[]{
			\makebox[.25\textwidth]{
\begin{tikzpicture}[scale = 6]
\tikzstyle{VertexStyle} = []
\tikzstyle{EdgeStyle} = []
\tikzstyle{labeledStyle}=[shape = circle, minimum size = 6pt, inner sep = 1.2pt, draw]
\tikzstyle{unlabeledStyle}=[shape = circle, minimum size = 6pt, inner sep = 1.2pt, draw, fill]
\Vertex[style = labeledThree, x = 0.550000011920929, y = 0.75, L = \Large {$~$}]{v0}
\Vertex[style = labeledThree, x = 0.649999976158142, y = 0.75, L = \Large {$~$}]{v1}
\Vertex[style = labeledThree, x = 0.75, y = 0.849999994039536, L = \Large {$~$}]{v2}
\Vertex[style = labeledThree, x = 0.850000023841858, y = 0.75, L = \Large {$~$}]{v3}
\Vertex[style = labeledThree, x = 0.949999988079071, y = 0.75, L = \Large {$~$}]{v4}
\Vertex[style = labeledTwo, x = 0.850000023841858, y = 0.650000005960464, L =
\Large {$\>$}]{v5}
\Vertex[style = labeledTwo, x = 0.649999976158142, y = 0.650000005960464, L =
\Large {$\>$}]{v6}
\Vertex[style = labeledThree, x = 0.449999988079071, y = 0.75, L = \Large {$~$}]{v7}
\Vertex[style = labeledTwo, x = 0.449999988079071, y = 0.650000005960464, L =
\Large {$\>$}]{v8}
\Vertex[style = labeledThree, x = 1.04999995231628, y = 0.75, L = \Large {$~$}]{v9}
\Edge[label = \Large {}, labelstyle={auto=right, fill=none}](v1)(v0)
\Edge[label = \tiny {}, labelstyle={auto=right, fill=none}](v1)(v2)
\Edge[label = \tiny {}, labelstyle={auto=right, fill=none}](v3)(v2)
\Edge[label = \tiny {}, labelstyle={auto=right, fill=none}](v3)(v4)
\Edge[label = \tiny {}, labelstyle={auto=right, fill=none}](v3)(v5)
\Edge[label = \tiny {}, labelstyle={auto=right, fill=none}](v5)(v0)
\Edge[label = \tiny {}, labelstyle={auto=right, fill=none}](v6)(v1)
\Edge[label = \tiny {}, labelstyle={auto=right, fill=none}](v7)(v0)
\Edge[label = \tiny {}, labelstyle={auto=right, fill=none}](v7)(v8)
\Edge[label = \tiny {}, labelstyle={auto=right, fill=none}](v9)(v4)
\Edge[label = \tiny {}, labelstyle={auto=right, fill=none}](v9)(v8)
\end{tikzpicture}
}}
				
\subfloat[]{
\makebox[.25\textwidth]{
\begin{tikzpicture}[scale = 6]
\tikzstyle{VertexStyle} = []
\tikzstyle{EdgeStyle} = []
\tikzstyle{labeledStyle}=[shape = circle, minimum size = 6pt, inner sep = 1.2pt, draw]
\tikzstyle{unlabeledStyle}=[shape = circle, minimum size = 6pt, inner sep = 1.2pt, draw, fill]
\Vertex[style = labeledThree, x = 0.649999976158142, y = 0.75, L = \Large {$~$}]{v0}
\Vertex[style = labeledThree, x = 0.75, y = 0.75, L = \Large {$~$}]{v1}
\Vertex[style = labeledTwo, x = 0.75, y = 0.650000005960464, L = \Large {$\>$}]{v2}
\Vertex[style = labeledTwo, x = 0.649999976158142, y = 0.650000005960464, L = \Large {$\>$}]{v3}
\Vertex[style = labeledThree, x = 0.25, y = 0.75, L = \Large {$~$}]{v4}
\Vertex[style = labeledThree, x = 0.35, y = 0.75, L = \Large {$~$}]{v5}
\Vertex[style = labeledThree, x = 0.45, y = 0.75, L = \Large {$~$}]{v6}
\Vertex[style = labeledTwo, x = 0.35, y = 0.65, L = \Large {$\>$}]{v7}
\Vertex[style = labeledThree, x = 0.55, y = 0.85, L = \Large {$~$}]{v8}
\Edge[label = \Large {}, labelstyle={auto=right, fill=none}](v1)(v0)
\Edge[label = \Large {}, labelstyle={auto=right, fill=none}](v2)(v1)
\Edge[label = \Large {}, labelstyle={auto=right, fill=none}](v3)(v0)
\Edge[label = \Large {}, labelstyle={auto=right, fill=none}](v3)(v4)
\Edge[label = \Large {}, labelstyle={auto=right, fill=none}](v4)(v5)
\Edge[label = \tiny {}, labelstyle={auto=right, fill=none}](v6)(v5)
\Edge[label = \tiny {}, labelstyle={auto=right, fill=none}](v7)(v5)
\Edge[label = \tiny {}, labelstyle={auto=right, fill=none}](v8)(v0)
\Edge[label = \tiny {}, labelstyle={auto=right, fill=none}](v8)(v6)
\end{tikzpicture}
}}
\subfloat[]{
\makebox[.25\textwidth]{
\begin{tikzpicture}[scale = 6]
\tikzstyle{VertexStyle} = []
\tikzstyle{EdgeStyle} = []
\tikzstyle{labeledStyle}=[shape = circle, minimum size = 6pt, inner sep = 1.2pt, draw]
\tikzstyle{unlabeledStyle}=[shape = circle, minimum size = 6pt, inner sep = 1.2pt, draw, fill]
\Vertex[style = labeledThree, x =
0.600000023841858, y = 0.75, L = \Large {$~$}]{v0}
\Vertex[style = labeledThree, x =
0.699999988079071, y = 0.75, L = \Large {$~$}]{v1}
\Vertex[style = labeledTwo, x =
0.699999988079071, y = 0.650000005960464, L = \Large {$\>$}]{v2}
\Vertex[style = labeledTwo, x =
0.600000023841858, y = 0.650000005960464, L = \Large {$\>$}]{v3}
\Vertex[style = labeledThree, x =
0.300000011920929, y = 0.75, L = \Large {$~$}]{v4}
\Vertex[style = labeledThree, x =
0.400000005960464, y = 0.75, L = \Large {$~$}]{v5}
\Vertex[style = labeledTwo, x =
0.400000005960464, y = 0.650000005960464, L = \Large {$\>$}]{v6}
\Vertex[style = labeledThree, x = 0.5, y =
0.849999994039536, L = \Large {$~$}]{v7}
\Edge[label = \Large {}, labelstyle={auto=right, fill=none}](v1)(v0)
\Edge[label = \tiny {}, labelstyle={auto=right, fill=none}](v2)(v1)
\Edge[label = \tiny {}, labelstyle={auto=right, fill=none}](v3)(v0)
\Edge[label = \tiny {}, labelstyle={auto=right, fill=none}](v4)(v2)
\Edge[label = \tiny {}, labelstyle={auto=right, fill=none}](v5)(v4)
\Edge[label = \tiny {}, labelstyle={auto=right, fill=none}](v5)(v6)
\Edge[label = \tiny {}, labelstyle={auto=right, fill=none}](v7)(v0)
\Edge[label = \tiny {}, labelstyle={auto=right, fill=none}](v7)(v5)
\end{tikzpicture}
}}
\subfloat[]{
\makebox[.25\textwidth]{
\begin{tikzpicture}[scale = 6]
\tikzstyle{VertexStyle} = []
\tikzstyle{EdgeStyle} = []
\tikzstyle{labeledStyle}=[shape = circle, minimum size = 6pt, inner sep = 1.2pt, draw]
\tikzstyle{unlabeledStyle}=[shape = circle, minimum size = 6pt, inner sep = 1.2pt, draw, fill]
\Vertex[style = labeledThree, x = 0.65, y = 0.75, L = \Large {$~$}]{v0}
\Vertex[style = labeledThree, x = 0.75, y = 0.75, L = \Large {$~$}]{v1}
\Vertex[style = labeledThree, x = 0.85, y = 0.75, L = \Large {$~$}]{v2}
\Vertex[style = labeledTwo, x = 0.75, y = 0.65, L = \Large {$\>$}]{v3}
\Vertex[style = labeledTwo, x = 0.65, y = 0.65, L = \Large {$\>$}]{v4}
\Vertex[style = labeledTwo, x = 0.85, y = 0.65, L = \Large {$\>$}]{v5}
\Vertex[style = labeledThree, x = 0.35, y = 0.75, L = \Large {$~$}]{v6}
\Vertex[style = labeledThree, x = 0.45, y = 0.75, L = \Large {$~$}]{v7}
\Vertex[style = labeledTwo, x = 0.45, y = 0.65, L = \Large {$\>$}]{v8}
\Vertex[style = labeledThree, x = 0.55, y =
0.849999994039536, L = \Large {$~$}]{v9}
\Edge[label = \Large {}, labelstyle={auto=right, fill=none}](v1)(v0)
\Edge[label = \tiny {}, labelstyle={auto=right, fill=none}](v1)(v2)
\Edge[label = \tiny {}, labelstyle={auto=right, fill=none}](v3)(v1)
\Edge[label = \tiny {}, labelstyle={auto=right, fill=none}](v4)(v0)
\Edge[label = \tiny {}, labelstyle={auto=right, fill=none}](v5)(v2)
\Edge[label = \tiny {}, labelstyle={auto=right, fill=none}](v6)(v5)
\Edge[label = \tiny {}, labelstyle={auto=right, fill=none}](v7)(v6)
\Edge[label = \tiny {}, labelstyle={auto=right, fill=none}](v7)(v8)
\Edge[label = \tiny {}, labelstyle={auto=right, fill=none}](v9)(v0)
\Edge[label = \tiny {}, labelstyle={auto=right, fill=none}](v9)(v7)
\end{tikzpicture}
}}
\subfloat[]{
\makebox[.25\textwidth]{
\begin{tikzpicture}[scale = 6]
\tikzstyle{VertexStyle} = []
\tikzstyle{EdgeStyle} = []
\tikzstyle{labeledStyle}=[shape = circle, minimum size = 6pt, inner sep = 1.2pt, draw]
\tikzstyle{unlabeledStyle}=[shape = circle, minimum size = 6pt, inner sep = 1.2pt, draw, fill]
\Vertex[style = labeledThree, x = 0.70, y = 0.75, L = \Large {$~$}]{v0}
\Vertex[style = labeledThree, x = 0.80, y = 0.75, L = \Large {$~$}]{v1}
\Vertex[style = labeledThree, x = 0.90, y = 0.75, L = \Large {$~$}]{v2}
\Vertex[style = labeledTwo, x = 0.80, y = 0.65, L = \Large {$\>$}]{v3}
\Vertex[style = labeledTwo, x = 0.70, y = 0.65, L = \Large {$\>$}]{v4}
\Vertex[style = labeledTwo, x = 0.90, y = 0.65, L = \Large {$\>$}]{v5}
\Vertex[style = labeledThree, x = 0.30, y = 0.75, L = \Large {$~$}]{v6}
\Vertex[style = labeledThree, x = 0.40, y = 0.75, L = \Large {$~$}]{v7}
\Vertex[style = labeledThree, x = 0.5, y = 0.75, L = \Large {$~$}]{v8}
\Vertex[style = labeledTwo, x = 0.5, y = 0.65, L = \Large {$\>$}]{v9}
\Vertex[style = labeledThree, x = 0.60, y = 0.85, L = \Large {$~$}]{v10}
\Edge[label = \Large {}, labelstyle={auto=right, fill=none}](v1)(v0)
\Edge[label = \Large {}, labelstyle={auto=right, fill=none}](v1)(v2)
\Edge[label = \Large {}, labelstyle={auto=right, fill=none}](v3)(v1)
\Edge[label = \Large {}, labelstyle={auto=right, fill=none}](v4)(v0)
\Edge[label = \Large {}, labelstyle={auto=right, fill=none}](v5)(v2)
\Edge[label = \Large {}, labelstyle={auto=right, fill=none}](v6)(v5)
\Edge[label = \Large {}, labelstyle={auto=right, fill=none}](v6)(v7)
\Edge[label = \tiny {}, labelstyle={auto=right, fill=none}](v8)(v7)
\Edge[label = \tiny {}, labelstyle={auto=right, fill=none}](v8)(v9)
\Edge[label = \tiny {}, labelstyle={auto=right, fill=none}](v10)(v0)
\Edge[label = \tiny {}, labelstyle={auto=right, fill=none}](v10)(v8)
\end{tikzpicture}
}}

\subfloat[]{
\makebox[.25\textwidth]{
\begin{tikzpicture}[scale = 6]
\tikzstyle{VertexStyle} = []
\tikzstyle{EdgeStyle} = []
\tikzstyle{labeledStyle}=[shape = circle, minimum size = 6pt, inner sep = 1.2pt, draw]
\tikzstyle{unlabeledStyle}=[shape = circle, minimum size = 6pt, inner sep = 1.2pt, draw, fill]
\Vertex[style = labeledThree, x = 0.70, y = 0.70, L = \Large {$~$}]{v0}
\Vertex[style = labeledThree, x = 0.80, y = 0.70, L = \Large {$~$}]{v1}
\Vertex[style = labeledThree, x = 0.90, y = 0.70, L = \Large {$~$}]{v2}
\Vertex[style = labeledTwo, x = 0.80, y = 0.60, L = \Large {$\>$}]{v3}
\Vertex[style = labeledTwo, x = 0.70, y = 0.60, L = \Large {$\>$}]{v4}
\Vertex[style = labeledTwo, x = 0.90, y = 0.60, L = \Large {$\>$}]{v5}
\Vertex[style = labeledThree, x = 0.30, y = 0.70, L = \Large {$~$}]{v6}
\Vertex[style = labeledThree, x = 0.40, y = 0.70, L = \Large {$~$}]{v7}
\Vertex[style = labeledThree, x = 0.5, y = 0.70, L = \Large {$~$}]{v8}
\Vertex[style = labeledTwo, x = 0.30, y = 0.60, L = \Large {$\>$}]{v9}
\Vertex[style = labeledTwo, x = 0.5, y = 0.60, L = \Large {$\>$}]{v10}
\Vertex[style = labeledTwo, x = 0.40, y = 0.60, L = \Large {$\>$}]{v11}
\Vertex[style = labeledThree, x = 0.60, y = 0.80, L = \Large {$~$}]{v12}
\Edge[label = \Large {}, labelstyle={auto=right, fill=none}](v1)(v0)
\Edge[label = \Large {}, labelstyle={auto=right, fill=none}](v1)(v2)
\Edge[label = \tiny {}, labelstyle={auto=right, fill=none}](v3)(v1)
\Edge[label = \tiny {}, labelstyle={auto=right, fill=none}](v4)(v0)
\Edge[label = \tiny {}, labelstyle={auto=right, fill=none}](v5)(v2)
\Edge[label = \tiny {}, labelstyle={auto=right, fill=none}](v6)(v7)
\Edge[label = \tiny {}, labelstyle={auto=right, fill=none}](v8)(v7)
\Edge[label = \tiny {}, labelstyle={auto=right, fill=none}](v8)(v10)
\Edge[label = \tiny {}, labelstyle={auto=right, fill=none}](v9)(v6)
\Edge[label = \tiny {}, labelstyle={auto=right, fill=none}](v11)(v7)
\Edge[label = \tiny {}, labelstyle={auto=right, fill=none}](v12)(v0)
\Edge[label = \tiny {}, labelstyle={auto=right, fill=none}](v12)(v8)
\end{tikzpicture}}}
\caption{Extra reducible configurations.\label{extra-pic}}
\end{center}
\end{figure}
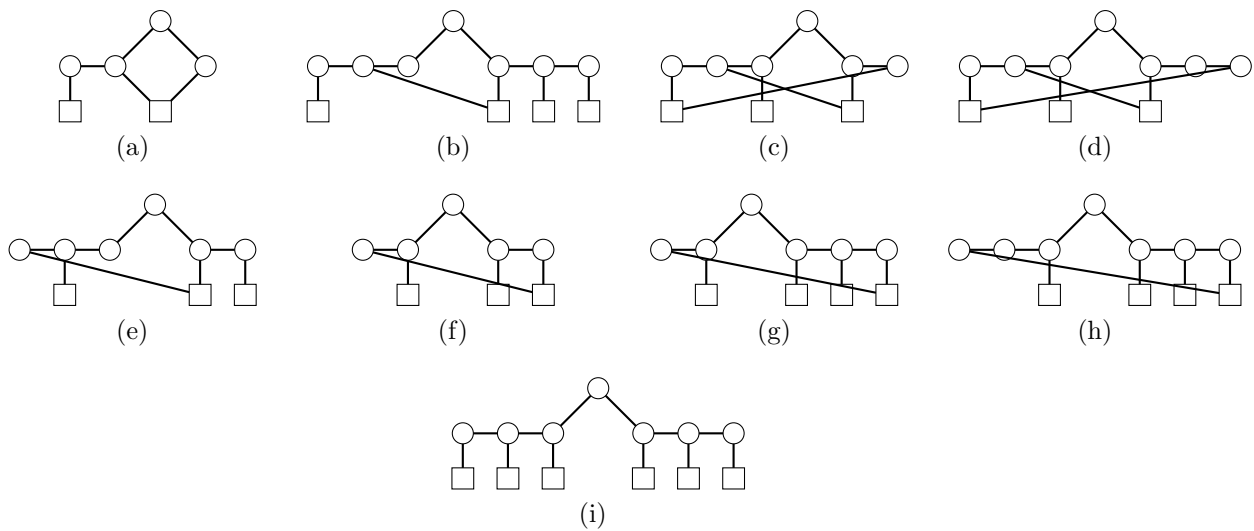

The reducible configurations in this paper were originally found by computer. 
The computer uses an abstract definition capturing the notion of ``colorable
after performing some Kempe swaps'', which frees it from considering an
embedding in an ambient critical graph.  This is called \emph{fixability} and
extends the idea in \cite{HallGame} from stars to arbitrary graphs.  The
computer is able to prove many reducibility results for which we have yet to find
short proofs.  Here we show how to use some of these reducible
configurations to further improve the bound on the average degree of
$3$-critical graphs. We give a larger survey of what can be proved with these
computer results in \cite{FixabilitySurvey}.

To conclude the paper we prove, modulo computer verification, that every
3-critical graph has $p(G)\le 9$.  This is the final piece needed for our
improved bound on $a(G)$ in Theorem~\ref{thm:main2}.

\begin{lem}\label{225}
The configurations in Figure \ref{extra-pic} are not subgraphs of any
3-critical graph.  
In particular, $p(G)\le 9$.
\end{lem}
\begin{proof}
The first statement is proved by computer.  We prove the second statement,
assuming the first.
Suppose that $v$ is a rich vertex with neighbors in poor fragments $H_1$, $H_2$,
$H_3$.  If at most one $H_i$ has order at least 3, then $p(v)\le 2+2+5=9$.
Suppose instead that $H_1$ and $H_2$ have order at least 3.
If $H_1$ and $H_2$ have no common 2-neighbors, then $G$ has a copy of
Figure~\ref{extra-pic}(i).  Otherwise, the poor fragments share one or more common
2-neighbors, and $G$ contains one of Figure~\ref{extra-pic}(a)--(h).
Choose a common 2-neighbor $w$ such that the shortest cycle $C$ through $v$ and
$w$ is as short as possible.  
If $C$ has length 4, then we have Figure~\ref{extra-pic}(a); 
if length 5, then one of Figures~\ref{extra-pic}(b--d); 
if length 6, then one of Figures~\ref{extra-pic}(e--f); 
if length 7, then Figure~\ref{extra-pic}(g);
if length 8, then Figure~\ref{extra-pic}(h).
The computer is able to generate proofs in \LaTeX, but at about 100 pages this
one is not a fun read:
\url{https://dl.dropboxusercontent.com/u/8609833/Papers/big%20tree.pdf}
\end{proof}


\section*{Acknowledgments}
Thanks to two anonymous referees for their constructive feedback.  In
particular, the extensive comments of one referee, subsequently identified as
Douglas Woodall, led to a significantly improved exposition, the insertion and
proof of Lemmas 8 and 10, a much simpler proof of Lemma~\ref{lem:fig2a}, and a
slightly stronger bound in Theorem~\ref{thm:main1} (showing $p(G)\le 10$,
rather than $p(G)\le 11$ as we had done).
\bibliographystyle{amsplain}
\bibliography{FixerBreaker}

\providecommand{\bysame}{\leavevmode\hbox to3em{\hrulefill}\thinspace}
\providecommand{\MR}{\relax\ifhmode\unskip\space\fi MR }
\providecommand{\MRhref}[2]{%
  \href{http://www.ams.org/mathscinet-getitem?mr=#1}{#2}
}
\providecommand{\href}[2]{#2}
\begin{thebibliography}{10}

\bibitem{C&C}
D.~Cariolaro and G.~Cariolaro, \emph{Colouring the petals of a graph},
  Electron. J. Combin. \textbf{10} (2003), Research Paper 6, 11 pp.

\bibitem{FixabilitySurvey}
D.W. Cranston and L.~Rabern, \emph{{Edge-coloring via Fixable Subgraphs}},
  arXiv preprint, available at \url{https://arxiv.org/abs/1507.05600} (2015).

\bibitem{FioriniWilson}
S.~Fiorini and R.J. Wilson, \emph{On the chromatic index of a graph. {II}},
  Combinatorics ({P}roc. {B}ritish {C}ombinatorial {C}onf., {U}niv. {C}oll.
  {W}ales, {A}berystwyth, 1973), Cambridge Univ. Press, London, 1974,
  pp.~37--51. London Math. Soc. Lecture Note Ser., No. 13.

\bibitem{FioriniW77}
\bysame, \emph{Edge-colourings of graphs}, Pitman, London; distributed by
  Fearon-Pitman Publishers, Inc., Belmont, Calif., 1977, Research Notes in
  Mathematics, No. 16.

\bibitem{Jakobsen73}
I.T. Jakobsen, \emph{Some remarks on the chromatic index of a graph}, Arch.
  Math. (Basel) \textbf{24} (1973), 440--448.

\bibitem{Jakobsen74}
\bysame, \emph{On critical graphs with chromatic index {$4$}}, Discrete Math.
  \textbf{9} (1974), 265--276.

\bibitem{HallGame}
L.~Rabern, \emph{A game generalizing {H}all's theorem}, Discrete Math.
  \textbf{320} (2014), 87--91.

\bibitem{StiebitzSTF12}
M.~Stiebitz, D.~Scheide, B.~Toft, and L.M. Favrholdt, \emph{Graph edge
  coloring: {V}izing's theorem and {G}oldberg's conjecture}, Wiley Series in
  Discrete Mathematics and Optimization, John Wiley \& Sons, Inc., Hoboken, NJ,
  2012.

\bibitem{Vizing65}
V.G. Vizing, \emph{The chromatic class of a multigraph}, Kibernetika (Kiev)
  \textbf{1965} (1965), no.~3, 29--39.

\bibitem{Vizing68}
\bysame, \emph{Some unsolved problems in graph theory}, Uspehi Mat. Nauk
  \textbf{23} (1968), no.~6 (144), 117--134.

\bibitem{Woodall08}
D.R. Woodall, \emph{The average degree of an edge-chromatic critical graph},
  Discrete Math. \textbf{308} (2008), no.~5-6, 803--819.

\end{thebibliography}
\end{document}